\documentclass{amsart}

\usepackage{amssymb}
\usepackage[utf8]{inputenc}
\usepackage{xcolor}
\usepackage{mathtools}
\usepackage{hyperref}
\usepackage{picture}
\usepackage{ifthen}
\usepackage[inline]{enumitem}
\usepackage{mathrsfs}
\usepackage[all]{xy}

\newcommand{\Dchaintwo}[3]{%
\setlength{\unitlength}{.08cm}
\rule[-3\unitlength]{0pt}{8\unitlength}
\begin{picture}(14,5)(0,3)
\put(1,2){\circle{2}}
\put(2,2){\line(1,0){10}}
\put(13,2){\circle{2}}
\put(1,5){\makebox[0pt]{\scriptsize #1}}
\put(7,4){\makebox[0pt]{\scriptsize #2}}
\put(13,5){\makebox[0pt]{\scriptsize #3}}
\end{picture}
}
\newcommand{\ad}{\operatorname{ad}}
\newcommand{\ord}{\operatorname{ord}}
\newcommand{\Aff}{\operatorname{Aff}}
\newcommand{\Aut}{\operatorname{Aut}}
\newcommand{\car}{\operatorname{char}}
\newcommand{\End}{\operatorname {End}}

\newcommand{\gr}{\operatorname {gr}}
\newcommand{\id}{\operatorname {id}}
\newcommand{\Ind}{\operatorname{Ind}}
\newcommand{\Inn}{\operatorname{Inn}}

\newcommand{\Rad}{\operatorname{Rad}}

\newcommand{\lspan}{\operatorname{span}}
\newcommand{\supp}{\operatorname{supp}}

\newcommand{\bomega}{w}
\newcommand{\aut}{\mathfrak{t}}
\newcommand{\tetra}{\mathscr{T}}

\newcommand{\BG}[1]{\mathbb{B}_{#1}}
\newcommand{\mbG}{\mathbb{G}}

\newcommand{\FK}{\tt{FK}}

\newcommand{\cB }{\mathcal{B}}

\newcommand{\cF }{\mathcal{F}}

\newcommand{\hurw}{\mathcal{H}}

\newcommand{\cO }{\mathcal{O}}
\newcommand{\Oc}{\mathcal{O}}
\newcommand{\cR }{\mathcal{R}}
\newcommand{\Ss}{\mathcal{S}}

\newcommand{\cW}{\mathcal{W}}
\newcommand{\cX }{\mathcal{X}}
\newcommand{\co }[1]{\operatorname {co}\,#1}

\newcommand{\fie}{\Bbbk}
\newcommand{\Fie}{\mathbb{\fie }}

\newcommand{\fm }{\mathfrak{m}}

\newcommand{\NA}{\mathscr{B}}
\newcommand{\ndF}{\mathbb{F}}
\newcommand{\ndN}{\mathbb{N}}

\newcommand{\ndZ}{\mathbb{Z}}
\newcommand{\ot }{\otimes }
\newcommand{\SG }[1]{\mathbb{S}_{#1}}
\newcommand{\AG }[1]{\mathbb{A}_{#1}}

\newcommand{\yd}[1]{\prescript{#1}{#1}{\mathcal{YD}}}

\newcommand{\Dchainthree}[6]{
\rule[-3\unitlength]{0pt}{8\unitlength}
\begin{picture}(26,5)(0,3)
\put(1,2){\ifthenelse{\equal{#1}{l}}{\circle*{2}}{\circle{2}}}
\put(2,2){\line(1,0){10}}
\put(13,2){\ifthenelse{\equal{#1}{m}}{\circle*{2}}{\circle{2}}}
\put(14,2){\line(1,0){10}}
\put(25,2){\ifthenelse{\equal{#1}{r}}{\circle*{2}}{\circle{2}}}
\put(1,5){\makebox[0pt]{\scriptsize #2}}
\put(7,4){\makebox[0pt]{\scriptsize #3}}
\put(13,5){\makebox[0pt]{\scriptsize #4}}
\put(19,4){\makebox[0pt]{\scriptsize #5}}
\put(25,5){\makebox[0pt]{\scriptsize #6}}
\end{picture}}

\numberwithin{equation}{section}

\newtheorem{thm}{Theorem}[section]
\newtheorem{lem}[thm]{Lemma}
\newtheorem{pro}[thm]{Proposition}
\newtheorem{cor}[thm]{Corollary}

\theoremstyle{definition}

\newtheorem{exa}[thm]{Example}

\theoremstyle{remark}
\newtheorem{rem}[thm]{Remark}
\newtheorem{step}{Step}

\newtheorem{case}{Case}

\newcommand{\vbd}{U}
\newcommand{\fiet}{\Bbbk^{\times}}
\newcommand{\bq}{\mathfrak q}

\title[]{Pointed Hopf algebras of odd dimension 
and Nichols algebras over solvable groups}

\author{N. Andruskiewitsch}
\author{I. Heckenberger}
\author{L. Vendramin}

\address[N.~Andruskiewitsch]{Facultad de Matem\'atica, Astronom\'ia y F\'isica,
Universidad Nacional de C\'ordoba. CIEM -- CONICET. 
Medina Allende s/n (5000) Ciudad Universitaria, C\'ordoba, Argentina}
\email{nicolas.andruskiewitsch@unc.edu.ar}

\address[I. Heckenberger]{Philipps-Universit\"at Marburg,
FB Mathematik und Informatik,
Hans-Meer\-wein-Stra\ss e,
35032 Marburg, Germany.}
\email{heckenberger@mathematik.uni-marburg.de}

\address[L. Vendramin]{Department of Mathematics and Data
Science, Vrije Universiteit Brussel, Pleinlaan 2, 1050 Brussels, Belgium}
\email{Leandro.Vendramin@vub.be}


\begin{document}

\begin{abstract}
We classify finite-dimensional Nichols algebras of Yetter-Drinfeld modules
with indecomposable support over finite solvable groups in characteristic 0,
using a variety of methods including reduction to positive characteristic.
As a consequence, all Nichols algebras over groups of odd order are of diagonal type, which allows us to describe all pointed Hopf algebras of odd dimension.
\end{abstract}

\maketitle

\section{Introduction}\label{sec:intro}

\subsection{Pointed Hopf algebras of odd dimension}
A  Hopf algebra is pointed if every simple comodule is one-dimensional.
The class of pointed Hopf algebras encompasses the quantized enveloping algebras
of Lie (super)algebras and their finite-dimensional versions.
Two basic invariants of a pointed Hopf algebra $H$ are:

\begin{itemize} [leftmargin=*]\renewcommand{\labelitemi}{$\circ$}
\item the group $G(H)$ of group-like elements,

\item the infinitesimal braiding $V$ of $H$, a Yetter-Drinfeld module  
over $\fie G(H)$, see \cite{MR1913436}.
\end{itemize}

These invariants and, fundamentally, the Nichols algebra $\NA(V)$ constitute
the core of the approach to the classification of pointed Hopf algebras
proposed in \cite{MR1913436}. One of the results obtained with this approach 
is the classification of the finite-dimensional pointed Hopf algebras $H$ in characteristic 0
such that $G(H)$ is abelian and has order coprime to $210$ \cite{MR2630042}. 
Here we extend the mentioned result
as follows.

\begin{thm}\label{mainthm:pointed-odd}
Let $\fie$ be an algebraically closed field of characteristic 0.
Let $H$ be a pointed Hopf algebra such that $d = \dim H$ is odd.
Then:

\begin{enumerate}[leftmargin=*,label=\rm{(\alph*)}]

\item\label{item:pointed-odd-0}  
the infinitesimal braiding $V$ of $H$ is a braided vector space of diagonal type;

\medbreak
\item\label{item:pointed-odd-1} 
$H$ is generated by group-like and skew-primitive elements;

\medbreak
\item\label{item:pointed-odd-2} 
$H$ is a cocycle deformation of $\NA(V) \# \fie G(H)$;

\medbreak
\item\label{item:pointed-odd-3} 
If $\gcd(d, 105) = 1$, then the defining relations of $H$ are 
quantum Serre relations, linking relations and powers of the root vectors.
\end{enumerate}
\end{thm}

In other words, Theorem \ref{mainthm:pointed-odd} characterizes pointed Hopf 
algebras of odd dimension, that can be explicitly described up to group-theoretical information. 

\medbreak
Notice that if $\gcd(d,3) = 1$, then $V$ is of Cartan type;
see Remark \ref{rem:diagonal-odd-order}.

\medbreak
Assume that \ref{item:pointed-odd-0} is proven. Then
\ref{item:pointed-odd-1} follows from  \cite{MR3181554}, or \cite[Theorem 5.5]{MR2630042} when  $\gcd(d, 105) = 1$; 
\ref{item:pointed-odd-2} follows from \cite{MR3908852};  \ref{item:pointed-odd-3} follows from \cite[Theorem 6.2]{MR2630042}.

\medbreak
Thus, the main contribution of the present paper to 
Theorem~\ref{mainthm:pointed-odd} is to prove 
\ref{item:pointed-odd-0}, which is Theorem~
\ref{thm:solvable-210-order-abelian-or-typeC}. 
The proof of this last result follows from 
Theorem \ref{mainthm:nichols-solvable}, discussed in the next subsection,
since $G(H)$ is solvable by the Feit-Thompson Theorem.
Indeed, if $\dim H$ is odd, then $\vert G(H) \vert$ is odd by the 
Nichols-Zoeller Theorem.
Conversely, if $G$ is a group with odd order, then
any finite-dimensional pointed Hopf algebra $H$ with $G(H) \simeq G$ 
has odd dimension since the finite-dimensional Nichols algebras over $G$
are of diagonal type hence they have odd dimension (see Remark~\ref{rem:diagonal-odd-order});
in fact Theorem \ref{thm:solvable-210-order-abelian-or-typeC} only requires that $G$ has odd order.

\subsection{Nichols  algebras over solvable groups}
The study of the finite-dimensional Nichols algebras of Yetter-Drinfeld modules 
over finite non-abelian groups has a long history that we summarize briefly.
Let us assume first that $\car \fie = 0$.

\medbreak
The first examples are the so called Fomin-Kirillov algebras
$\FK_n$ associated to transpositions in $\SG{n}$ where $n = 3,4,5$;
they were introduced in \cite{MR1667680} in 1995 and realized later
as braided Hopf algebras by Fomin and Procesi. These algebras
and two relatives for $n = 4, 5$ were
discovered and interpreted as Nichols algebras independently by Milinski and
Schneider around 1996,
published later in \cite{MR1800714}.

\medbreak
Between 2000 and 2001, Gra\~na discovered six more examples, appeared in \cite{MR1800709,MR1994219}:
one related to  (any of) the conjugacy classes of $3$-cycles in $\AG{4}$
(which are isomorphic as racks and denoted $\tetra$),
two related to affine racks with $5$ elements, 
two related to affine racks with $7$ elements and 
one related to the class of $4$-cycles in $\mathbb S_4$.

\medbreak
In  2012, a new example related to $\tetra$ was presented in \cite{MR2891215}. The Nichols algebras above are discussed in Examples~\ref{exa:affine}, \ref{exa:T}, \ref{exa:S4a}
and \ref{exa:S4b} and Remark~\ref{rem:FK5}, 
see Subsection~\ref{subsec:main-examples};  
they all arise from solvable groups except $\FK_5$ and its relative. 

\medbreak
In prime characteristic $p$, the Nichols algebras  above remain
finite-dimensional, 
as is apparent from Section \ref{sec:-char0}; in fact they have the same dimension as in characteristic 0 except one of the related with $\tetra$
when $p=2$, see \cite{MR2803792}. 
Besides there is an example associated to transpositions in $\SG{3}$ 
of dimension 432 when $p=2$, see \cite{MR2891215}.

\medbreak
Since 2006, two lines of research on 
Nichols algebras over finite
groups were pursued: the classification of the finite-dimensional ones corresponding to
semisimple, not simple, Yetter-Drinfeld modules, 
achieved in \cite{MR3605018,MR3656477} (valid in arbitrary characteristic); and the determination of
the Nichols algebras over simple groups with infinite dimension,
see \cite{MR4664370,MR2786171,MR2745542} and references therein.

\medbreak
In 2023, a new point of view was introduced in \cite{MR4729697}, where 
the finite-dimensional Nichols algebras related to a simple rack with
a prime number of elements were classified using several different techniques, including reduction to positive characteristic;
there are no new examples.

\medbreak Here is the main result of this paper, see Theorem~\ref{thm:char_zero}:

\begin{thm}\label{mainthm:nichols-solvable}
Let $\fie$ be a field of characteristic 0.
Let $G$ be a finite group whose solvable radical contains properly the hypercenter.
Let $V$ be a Yetter-Drinfeld module over $\fie G$ such that
\begin{itemize}[leftmargin=*]\renewcommand{\labelitemi}{$\circ$}
\item the support of $V$ is a conjugacy class of $G$ generating $G$,
\item the dimension of the Nichols algebra $\NA(V)$ is finite.
\end{itemize}
Then $V$ is absolutely simple  
and isomorphic to one of the Yetter-Drinfeld modules
of Examples~\emph{\ref{exa:affine}, \ref{exa:T}, \ref{exa:S4a}} 
and \emph{\ref{exa:S4b}}
in Subsection~\ref{subsec:main-examples}. 
\end{thm}

We spell out the consequence for solvable groups
in Theorem~\ref{thm:char_zero-solvable}, that implies in turn Theorem
\ref{thm:solvable-210-order-abelian-or-typeC} for groups of odd order.

\medbreak
We observe that the classification of the finite-dimensional Nichols algebras
over solvable groups of even order does not present major mathematical obstacles
and follows from \cite{MR3605018,MR3656477} and Theorem~\ref{thm:char_zero-solvable}. 
The classification of finite-dimensional Nichols algebras
over non-solvable groups remains open: 
for instance, it is not known whether the Fomin-Kirillov algebra $\FK_6$
has finite dimension or not.

\subsection{The roadmap of the proof of Theorem~\ref{mainthm:nichols-solvable}}\label{subsec:roadmap}

Let $G$ and $V$ be as in Theorem~\ref{mainthm:nichols-solvable}. 
Then $V$ is absolutely simple,
as a consequence of \cite{MR3656477} (see Corollary~\ref{cor:absolutely_irreducible}). Choose a Yetter-Drinfeld
order $V_R$ of $V$ for a suitable Dedekind domain $R$ 
and a prime number $p$ dividing the order of the Fitting subgroup of $G/Z^*(G)$. 
Let $\fm$ be a maximal ideal of $R$ such that $p \in \fm$, 
 $\ndF = R/ \fm$ and $V_{\ndF}= \ndF\ot _R V_R$. 
Then $V_{\ndF}$ is a Yetter-Drinfeld module over $\ndF G$
and $\NA (V_{\ndF})$  is  finite-dimensional because $\NA(V)$ is so.
Then Theorem~\ref{thm:char_p} allows us to recover all possibilities of $p$ and the reduction $V_{\ndF}$; there are only finitely many of them. 

\medbreak
Theorem~\ref{thm:char_p} is a classification result of Yetter-Drinfeld
modules $U$ with analogous hypothesis as Theorem~\ref{mainthm:nichols-solvable} 
but over a field of positive characteristic.
Concretely, $\NA(U)$ 
admits a decreasing Hopf algebra filtration such that the associated graded Hopf algebra
has a direct sum of at least two simple Yetter-Drinfeld submodules of skew-primitive 
elements (because a group-like whose order is the characteristic of the field gives rise to a primitive element). The corresponding Nichols algebra is then finite-dimensional. Elaborating
on several classification results in positive characteristics collected in Section~\ref{sec:prime-char} (e.g. \cite{MR3656477,MR3625122,MR4099895,MR3313687})  
we conclude that $U$ isomorphic to one of the few examples. 

\medbreak
Finally, in order to determine all possibilities for $V$, we compare the Hurwitz orbits in the degree two part of $\NA (V)$ with their reductions modulo suitable primes. By Lemma~\ref{lem:concrete_new_trick}, if the ranks of these Hurwitz orbits differ and some technical assumptions are fulfilled, then $\NA (V)$ is shown to be infinite-dimensional. The remaining Nichols algebras have been already known to be finite-dimensional, and form the list of examples in Theorem~\ref{mainthm:nichols-solvable}.

\subsection{Affine simple racks}\label{subsec:simple-racks-intro}
For an introduction to racks we refer to Section~\ref{subsec:racks}. 
The problem of determining the simple racks $X$ for which there exist
a finite group $G$ and $V \in \yd{\fie G}$ such that
$\dim \NA(V) < \infty$ and $\supp V = X$ was raised in \cite{MR2786171}.
Assume  that $\car \fie = 0$. See \cite{MR4664370} and references therein for partial results on this question. The problem was solved in \cite{MR4729697}
when $\vert X \vert$ is prime. Here we give the answer when $X$ is 
 affine as another consequence of Theorem \ref{mainthm:nichols-solvable}.

\begin{thm} \label{thm:list-affine-simple-racks}
Let $X$ be an affine simple rack such that there exist
a finite group $G$ and $V \in \yd{\fie G}$ such that 
$\dim \NA(V) < \infty$ and $\supp V \simeq X$ as racks. 
Then $X \simeq \Aff(\varGamma, \aut)$ with $\varGamma = \ndF_q$, $q$ 
a power of a  prime, and 
$\aut$ being multiplication by $d \in \ndF_q^{\times}$, where 
\begin{align}\label{eq:list-affine-simple-racks}
(\ndF_q,d)\in\{(\ndF_3,2), (\ndF_4,\bomega), (\ndF_5,2),
(\ndF_5,3),(\ndF_7,3),(\ndF_7,5)\},
\end{align}
where $\bomega \in \mathbb F_4\backslash \{0,1\}$.
\end{thm}

See Subsection~\ref{subsec:simple-racks-appl} for the proof and some consequences.

\subsection{Notations}
Throughout $\fie$ is a field; assumptions about $\fie$ will be imposed when appropriate. The finite field with $q$ elements is denoted by $\ndF_q$.
If $R$ is a ring, then $R^{\times}$ denotes the group of units of $R$.

Let $G$ be a group. As usual, $\varGamma \leq G$ expresses that $\varGamma$ 
is a subgroup of $G$, while $\varGamma \lhd G$ means that it is normal.
The identity element of $G$ is denoted by $e$ and the center by $Z(G)$.
Given $g \in G$, its centralizer is denoted by $C_G(g)$ and its conjugacy class
by $\prescript{G}{}{g}$. More generally, the orbit of $x$ for an action of $G$ on a set $X$ is denoted by $\prescript{G}{}{x}$.
When $G$ is the symmetric group $\mathbb S_n$
and $g$ is a $d$-cycle, we set $\Oc^n_d = \prescript{\mathbb S_n}{}{g}$.

The order of $G$, or an element $g \in G$, or a rack $X$, is denoted by
$\vert G \vert$, or $\vert g \vert$, or $\vert X \vert$, respectively. 
The group algebra of $G$ is denoted by $\fie G$ and its augmentation ideal by $J_{G}$; $G$ is identified with the usual basis of $\fie G$.
For the functor $G \mapsto \fie G$, we denote 
again by $\xi: \fie G \to \fie H$
the algebra map induced by a morphism of groups $\xi: G \to H$.

All Hopf algebras are assumed with bijective antipode.
The notation for Hopf algebras is standard: $\Delta$, $\varepsilon$, $\Ss$
denote the comultiplication, the counit, the antipode, respectively.
We use the Sweedler notation for the comultiplication and the coactions.

\section{Yetter-Drinfeld modules and racks}\label{sec:YD-racks}

\subsection{Yetter-Drinfeld modules}\label{subsection:Yetter-Drinfeld}
\emph{Reference: \cite[Sections 3.10 and 4.3]{MR4164719}.}

Let $H$ be a Hopf algebra. 
The category of Yetter-Drinfeld modules over $H$ is denoted by $\yd{H}$.
This is a braided tensor category: if $M, N \in \yd{H}$, then the braiding
is given by 
\begin{align*}
c_{M,N}(m \otimes n) &= m_{(-1)} \cdot n \otimes m_{(0)}, & m \in M, &\ n \in N. 
\end{align*}
Particularly, if $M\in \yd{H}$, then $(M, c_{M,M})$ is a braided vector space, i.e.,
\begin{align*}
(c_{M,M} \otimes \id) (\id \otimes c_{M,M})(c_{M,M} \otimes \id) 
&= (\id \otimes c_{M,M}) (c_{M,M} \otimes \id)(\id \otimes c_{M,M}). 
\end{align*}

We shall need the following facts.

\begin{pro} \label{pro:functors}
Let $J$ be a Hopf ideal of $H$ and let
$\varphi:H\to H/J$ be the canonical Hopf algebra map. 
The following statements hold:

\medbreak
\begin{enumerate}[leftmargin=*,label=\rm{(\alph*)}]
\item There is a functor $\varphi_*:\yd H\to \yd{H/J}$,
\begin{align*}
V\mapsto H/J\ot_H V\simeq V/ \ad J_\varGamma (V),
\end{align*}
where $\varphi_*(V)$ is an $H/J$-module and $H/J$-comodule via
\begin{align*}
x\cdot (y\ot v)&=xy\ot v, &
\delta(x\otimes v)&=x_{(1)}\varphi(v_{(-1)})\Ss(x_{(3)})\otimes (x_{(2)}\otimes v_{(0)})
\end{align*}
for all $x,y\in H/J$ and $v\in V$.

\medbreak
\item There is a functor  
$\varphi^*: \yd{H/J}\to \yd H$,
\[M\mapsto H\square _{H/J} M, \]
where the cotensor product $H\square _{H/J}M$ is a left $H$-comodule and a left $H$-module via
\begin{align*}
\delta(h\ot m)&=\Delta(h)\ot m,
&
 h (h'\ot m) &=
h_{(1)}h'\Ss(h_{(3)})\ot h_{(2)}m. 
\end{align*}

\medbreak
\item The functor $\varphi_*$ is left adjoint to $\varphi^*$. 
\end{enumerate}
\end{pro}

\begin{proof}
See Propositions 4.5.1 and 4.5.2 and Corollary~4.5.3 of \cite{MR4164719}, respectively. 
\end{proof}

\medbreak
\subsection{Yetter-Drinfeld modules over groups}
Our primary interest is the case $H = \fie G$ where $G$ is a group;
then $V \in \yd{\fie G}$ means that $V = \bigoplus_{g \in G} V_g$ is a $G$-graded vector space and a $\fie G$-module such that
$g \cdot V_k = V_{gkg^{-1}}$ for all $g, k \in G$. The support of $V$ is
\begin{align*}
\supp V = \{g \in G: V_g \neq 0 \};
\end{align*}
it is a union of conjugacy classes of $G$. 

\medbreak
Assume that $G$ is finite. Let $g \in G$ and let $\chi: C_G(g) \to GL(V)$
be a representation of its centralizer. Then the induced module
\begin{align*} 
M(g, \chi) &= \Ind_{C_G(g)}^{G} V,
\end{align*}
with the natural grading from the identification 
of $G/C_G(g)$ with the conjugacy class of $g$, is an object in $\yd{\fie G}$. The representation $\chi $ is irreducible if and only if $M(g,\chi)$ is simple in $\yd{\fie G}$.
Furthermore, $M(g, \chi) \simeq M(g', \chi')$ if and only if $g$ and $g'$ are conjugated
and $\chi'$ arises from $\chi$ by conjugation.
This gives the classification of the simple objects in $\yd{\fie G}$.

\subsection{Change of group}
Throughout this subsection, $G$ denotes a finite group and $p$ is a prime number.
 We collect here results on finite groups 
 (well-known to experts) with applications to Yetter-Drinfeld modules that will be needed later. 

 \medbreak
 First we state a fact that follows from  Proposition \ref{pro:functors}.

\begin{lem}
\label{lem:easylowerstar} 
 Let 
$G$ be a group, $N \lhd G$, $\varphi:G\to G/N$ the canonical map, and 
$V \in \yd{\fie G}$. Assume that $\dim V < \infty$, that its support is a conjugacy class of $G$ and $N$ acts trivially on $V$. Let $g\in \supp V$. 
Then
\begin{align*} |\supp V|\cdot \dim V_g 
=\dim V
=\dim \varphi_*(V)
=|\supp \varphi_*(V)|\cdot \dim \varphi_*(V)_{gN}.
\end{align*}
Hence, if $\dim \varphi_*(V)_{gN}=1$ then $\supp V \simeq \supp \varphi_*(V)$ 
as racks. \qed \end{lem}

\medbreak
The next result expresses the well-known fact that the group algebra
functor from groups to Hopf algebras preserves exact sequences.

\begin{lem}
\label{lem:J}
Let  $\varGamma \lhd G$.
Then  $J = \fie G \cdot J_{\varGamma}$ is a Hopf ideal of $\fie G$; indeed it is the  kernel of the canonical Hopf algebra map
\begin{align*} \psi: \fie G\to \fie (G/\varGamma) \end{align*}
induced by the surjection $G \to G/\varGamma$. Hence
$\fie G/J \simeq \fie (G/\varGamma)$ as Hopf algebras.
Furthermore, if $ \car \fie =p>0$ and $\varGamma$ is a $p$-group, then 
$J$ is nilpotent. 
\end{lem}

\noindent \emph{Proof.} 
Since $J_\varGamma$ is a Hopf ideal of $\fie \varGamma$, 
$J$ is a coideal and a left ideal of $\fie G$. 
To see that $J$ is an ideal, observe, using that $\varGamma $ is a normal subgroup of $G$, that
\begin{align*}
(\gamma-1)g &= gg^{-1}(\gamma-1)g=g(g^{-1}\gamma g-1)\in J&
\text{for all } g &\in G.
\end{align*}

It is well-known that $J_\varGamma$ is nilpotent if 
$\car \fie = p>0$ and $\varGamma $ is a $p$-group, see 
\cite[Lemma 3.1.6]{MR0470211}. Hence $J$ is nilpotent, because
\begin{align*} 
J^k &= (\fie G\, J_\varGamma)^k=\fie G \, J_\varGamma^k & \text{for all } k &\in \ndN. \hspace{50pt}\qed \end{align*}

In the next proposition, given a group $\varGamma$ and a $\varGamma$-module $V$,
we use the notation 
\begin{align*}
H_0(\varGamma,V)\simeq  V/(J_{\varGamma} \cdot V),
\end{align*}
where $J\cdot V=\lspan_{\fie}\{h\cdot v\mid h\in J,v\in V\}$, 
for the $0$-th homology group. 
Also, given an action of $\varGamma$ on
a set $X$ we write
\begin{align*} X/\varGamma =\{\prescript{\varGamma}{}{x}
 \mid x\in X \} \end{align*}
for the set of all $\varGamma $-orbits  of $X$ and $[X/\varGamma]$ for a set of representatives of $X/\varGamma$ in $X$, so that
$X = \coprod_{x \in [X/\varGamma]}\prescript{\varGamma}{}{x}$.

\begin{pro} 
\label{pro:psistar}
Let $\varGamma \lhd G$, $\psi\colon G\to G/\varGamma$ be the canonical map,
and $V  \in \yd{\fie G}$; set $X = \supp V$. Then 
we have an isomorphism of $\fie (G/\varGamma)$-comodules
\begin{align*}
\psi_*(V)=H_0(\varGamma,V)\simeq 
\bigoplus_{x\in [X/\varGamma]} H_0(C_\varGamma(x),V_x).
\end{align*}
Moreover, if $J = \fie G\, J_{\varGamma}$ is nilpotent, then $H_0(C_\varGamma(x),V_x)$ is non-zero 
for all $x\in X$.
\end{pro}

\begin{proof}
If $x\in \supp V$, then 
$H_0(C_\varGamma(x),V_x)$ is a 
$\fie (G/\varGamma)$-comodule with coaction 
\begin{align*} 
\delta (v) &= \psi(x) \ot v & \text{ for all } v &\in H_0(C_\varGamma(x),V_x).
\end{align*} 
Let $\varGamma$ act on $X$ by conjugation.
The decomposition
\begin{align*} V=\bigoplus_{x\in [X/\varGamma ]} \ \bigoplus_{\gamma \in [\varGamma/C_{\varGamma}(x)]} V_{\gamma x \gamma^{-1}}
\end{align*}
is compatible with the coaction of $G/\varGamma$, and the $\fie (G/\varGamma)$-coaction is independent of the choice of representatives.
Then
\begin{align*} 
H_0(\varGamma,V)\simeq 
\bigoplus_{x\in [X/\varGamma ]} H_0\bigg(\varGamma ,\bigoplus_{\gamma \in [\varGamma/C_{\varGamma}(x)]} V_{\gamma x \gamma^{-1}}\bigg)
\simeq 
\bigoplus_{x\in [X/\varGamma]} H_0(C_\varGamma(x),V_x)
\end{align*}
as $\fie (G/\varGamma)$-comodules.

Finally, if $J$ is nilpotent, then for all $x\in X$, $J\cap \fie C_\varGamma(x)$ is a nilpotent ideal of $\fie C_\varGamma(x)$, and hence $H_0(C_\varGamma(x),V_x)$ is non-zero by Nakayama's lemma.
\end{proof}

We record a well-known lemma for further use.

\begin{lem}[{\cite[Lemma 3.11]{MR2426855}}]
\label{lem:minimal_normal}
If $M \lhd G$ is a minimal normal subgroup which is solvable, 
then $M$ is elementary abelian. \qed
\end{lem}

\subsubsection*{The hypercenter}
Let $Z_0(G) = \{e\}$, $Z_{n+1}(G) = \{x\in G\mid [x,G]\subseteq Z_{n}(G)\}$ be the upper central series of $G$, so that 
\begin{align*}
Z_0(G) &= \{e\} \leq Z_1(G) = Z(G) \leq Z_2(G) \leq \dots 
\end{align*}
The \emph{hypercenter of} $G$ is the limit $Z^*(G) =\cup_{n \in \ndN} Z_{n}(G)$; this is a  nilpotent
characteristic subgroup of $G$. Furthermore:

\medbreak
\begin{itemize}[leftmargin=*]\renewcommand{\labelitemi}{$\circ$}
\item $G$ is nilpotent if and only if $G = Z^*(G)$. 

\medbreak
\item $G/Z^*(G)$ is centerless. Let $\pi: G \to G/Z^*(G)$
be the canonical surjection. 

\end{itemize}

\subsubsection*{The  Fitting subgroup and  the $p$-core} \label{subsec:Fitting}
Recall that 

\begin{itemize}[leftmargin=*]\renewcommand{\labelitemi}{$\circ$}

\item the Fitting subgroup $F(G)$
is the unique largest normal nilpotent subgroup of $G$; 

\medbreak
\item the $p$-core $O_p(G)$ of $G$ is the intersection of 
the Sylow $p$-subgroups of $G$; it is the unique largest normal $p$-subgroup of $G$,
and the unique $p$-Sylow subgroup of the Fitting subgroup $F(G)$. 
\end{itemize} 
 
\begin{lem} 
\label{lem:hypercenter} The following statements hold: 
\begin{enumerate}[leftmargin=*,,label=\rm{(\alph*)}]
\item\label{item:hypercenter1} 
Let $M \lhd G$, $N \lhd G$ with $N \leq M$. If $N\leq Z^*(G)$ and $M/N$ is nilpotent, then $M$ is nilpotent.

\medbreak
\item\label{item:hypercenter2} 
Let $N \lhd G$ such that $N \leq Z^*(G)$. Then
$O_p(G/N)=O_p(G)N/N$.

\medbreak
\item\label{item:hypercenter3} 
$O_p(G)$ is contained in $Z^*(G)$ if and only if $O_p(G/Z^*(G))$ is trivial.
\end{enumerate}
\end{lem}

\begin{proof}
\ref{item:hypercenter1} 
Set $Z_j = Z_j(G)$. 
Since $N\leq Z^*(G)$, there exists $j\ge 1$ such that $N\leq Z_j$.
Let $(M_k)_{k\ge 1}$ be the lower central series of $M$, that is, 
\begin{align*}
M_1 &=M, & M_{k+1} &= [M,M_k],& k &\ge 1. 
\end{align*}
Since $M/N$ is nilpotent, there exists $r\ge 1$ such that $M_r\leq N$. 
Since
\[ [M,Z_i]\leq [G,Z_i]\leq Z_{i-1}\quad \text{for all $i\ge 1$} \]
and since $N\leq Z_j$, 
it follows that $M_{r+j}=\{e\}$.

\medbreak
\ref{item:hypercenter2} 
 Let $M$ be a normal $p$-subgroup of $G$. Then $MN/N$ is a  
 normal $p$-subgroup of $G/N$. Hence $O_p(G)N/N\leq O_p(G/N)$. (Here one does not need that $N\leq Z^*(G)$.)

\medbreak
Let now $M\leq G$ such that $N \leq M$, and assume that $M/N$ is a normal $p$-subgroup of $G/N$. Then $M$ is normal in $G$,
$|M|=p^n|N|$ for some $n\ge 0$,
and $M$ is nilpotent by \ref{item:hypercenter1}.
In particular, $M=\prod_{\text{$q$ prime}}O_q(M)$ and $O_p(M)\leq O_p(G)$. Moreover, $O_q(M)\leq N$ for all $q\ne p$. Hence $M/N\leq O_p(G)N/N$.

\medbreak
\ref{item:hypercenter3} follows from \ref{item:hypercenter2}  with $N=Z^*(G)$.
\end{proof}

\begin{lem}
\label{lem:Op(G)}
Let $N \lhd G$. Assume that
$N \leq Z^*(G)\cap O_p(G)$.  Then 
\begin{align*}
Z^*(G/N)&=Z^*(G)/N,\\
O_p(G/N)&=O_p(G)/N,\\
Z^*(G/N)\cap O_p(G/N)&=(Z^*(G)\cap O_p(G))/N.
\end{align*}
\end{lem}

\begin{proof} By definition, $Z_j(G)/N \leq Z_j(G/N)$
for any $j$; if $N \leq Z_r(G)$ for some $r$,
then $Z_j(G/N) \leq Z_{j + r}(G)/N$ and the first equality follows. The second is 
 Lemma \ref{lem:hypercenter} \ref{item:hypercenter2}.
The third is an elementary application of 
the correspondence theorem. 
\end{proof}

\begin{lem}
\label{lem:commutator}
Let $X$ be a conjugacy class
of $G$ generating $G$. Then $[G,G]$ acts transitively on $X$. 
\end{lem}

\begin{proof}
The group $[G,G]$ is generated by the elements 
$[g,y]=(gyg^{-1})y^{-1}$ with $g\in G$ and $y\in X$.
Since $G$ acts transitively on $X$, we conclude that
$[G,G]$ is generated by $\{xy^{-1}\mid x,y\in X\}$.
Let $Y$ be a non-empty 
subset of $X$ invariant under conjugation by $[G,G]$.
Then for all $x\in X$, $y\in Y$ one has
\begin{align*} xyx^{-1}=(xy^{-1})y(xy^{-1})^{-1}\in Y. \end{align*}
Hence $Y=X$, since $G$ is generated by $X$ and acts transitively on $X$.
\end{proof}

\begin{lem}
\label{lem:quotient_YD}
Assume that $G$ is non-abelian. Let
$V \in \yd{\fie G}$ be such that 
$\supp V$ is a conjugacy class of $G$ generating $G$, and let 
$N$ be the kernel of the representation of $G$ on $V$. 
Then $G/N$ is non-abelian.
\end{lem}

\begin{proof}
Let $X=\supp V$. Since $N$ acts trivially on $V$, the conjugation action of $N$ on $X$ is trivial.
 If $G/N$ is abelian, then $[G,G]\subseteq N$. In particular,
 $[G,G]$ acts trivially on $X$ by conjugation. 
 By Lemma~\ref{lem:commutator}, $[G,G]$ acts transitively on $X$. 
 Hence $|X|=1$. Since $X$ generates $G$, it follows that $G$ is abelian, a contradiction. 
\end{proof}

\begin{lem}
\label{lem:intersection}
Assume that $G$ is non-abelian and that $\car \fie = p$. Let $V  \in \yd{\fie G}$ be simple. Then 
$Z^*(G)\cap O_p(G)$ acts trivially on $V$. 
\end{lem}

\begin{proof}
Let $\rho:G\to GL(V)$ be the representation corresponding to the $G$-action on $V$, and let $N = Z^*(G)\cap O_p(G) \cap \ker \rho$. Our goal is to prove that $N=Z^*(G)\cap O_p(G)$.

Let $z \in Z(G/N) \cap O_p(G/N)$ and choose $g \in \supp V$.
Since $V$ is simple and $G$ is finite, $V_g$ is a simple $\fie C_{G}(g)N/N$-module. As $z\in C_{G}(g)N/N$ is  central, it acts by a scalar on $V_g$, by the Schur lemma. 
Since $\car \fie =p$ and $z\in O_p(G/N)$, it follows that $z$ acts trivially on $V_g$. Thus $Z(G/N)\cap O_p(G/N)$ is trivial. Arguing recursively, 
we see that $Z_j(G/N)\cap O_p(G/N)$ is trivial, hence so is
$Z^*(G/N)\cap O_p(G/N)$. 
By Lemma~\ref{lem:Op(G)}, $(Z^*(G)\cap O_p(G))/N$ is trivial, i.e., $(Z^*(G)\cap O_p(G)) \leq N$, as needed.
\end{proof}

\subsection{Racks}\label{subsec:racks}
A flexible formulation of the braided vector spaces arising from Yetter-Drinfeld modules
over finite groups is in terms of racks and cocycles. 
Recall that a rack is a nonempty set $X$ provided with a self-distributive 
binary operation $\triangleright$ such that the map $\phi_{x}: X \to X$,
$\phi_x (y)= x \triangleright y$ is bijective.
We also denote $x \triangleright^{-1} y = \phi^{-1}_x (y)$.
A rack $X$ is \emph{abelian}  if $x \triangleright y = y$ for any $x,y \in X$.
Any group is a rack with the conjugation operation $x \triangleright y = xyx^{-1}$;
the `forgetful' functor from the category of groups to that of racks
has a left adjoint $X \mapsto G_X$, where $G_X$ (called the \emph{enveloping group} of the rack $X$)
is defined as
\begin{align*}
G_X = \langle e_x: x\in X \mid
e_x e_y = e_{x\triangleright y} e_x, \ x,y\in X \rangle.
\end{align*}
We shall also consider other groups attached to a rack $X$:
\begin{itemize}[leftmargin=*]\renewcommand{\labelitemi}{$\circ$}
\item The derived subgroup $D_X = [G_X,G_X]$ of the enveloping group $G_X$.

\medbreak
\item The group $\Inn X$ of inner automorphisms of $X$
is the subgroup of $\SG{X}$ generated by the image of $\phi: X \to \SG{X}$,
$x \mapsto \phi_{x}$. Thus $\Inn X$ acts on $X$ by rack automorphisms. 
\end{itemize}

\medbreak
A rack $X$ is said to be be \emph{faithful} if the map $\phi: X \to \SG{X}$ is injective.

\begin{exa}\label{exa:inner} \cite[1.9]{MR1994219}
Let $X$ be a subrack of a finite group $G$. If $X$ generates $G$, then
$\Inn X \simeq G/Z(G)$; thus $\Inn X \simeq G$, if $G$ has trivial center.
\end{exa}

\begin{lem}\label{lem:enveloping-center}
Let $\Phi: G_X \to \Inn X$ be the map induced by $\phi: X \to \SG{X}$.
Then $Z(G_X) = \ker \Phi$.
Furthermore, if $X$ is finite, then $D_X$ is finite.
\end{lem}

\begin{proof}
The group $G_X$ acts on $X$ by $g\rightharpoonup x= \Phi(g)(x)$ for all $g\in G_{X}$ and $x\in X$.
Then $g e_{x} = e_{g\rightharpoonup x} g$ for all $g$ and $x$, 
implying that $Z(G_X) = \ker \Phi$.
If $X$ is finite, $G_X/Z(G_X)$ is finite, and hence so is $D_X$ 
 by Schur's theorem \cite[Theorem 5.32]{MR1307623}. 
\end{proof}

Two relevant classes of racks are the indecomposable and the simple ones:

\begin{itemize}[leftmargin=*]\renewcommand{\labelitemi}{$\circ$}
\item
A rack $X$ is \emph{indecomposable} if it is not a disjoint union of two (proper) subracks,
or equivalently if it is one orbit for the action of $\Inn X$.

\medbreak\item
A rack $X$ is \emph{simple} if it is not abelian,
 $\vert X \vert >1$ and for any surjective morphism of racks $\pi: X \to Y$,
 either $\vert Y \vert =1$ or else $\vert Y \vert= \vert X \vert $
(and $\pi$ is bijective). 
Simple racks are indecomposable.
\end{itemize}

We are interested in the subracks of finite groups, particularly the
conjugacy classes. Not all racks are like this; a necessary (but not sufficient)
condition is that $x \triangleright x = x$ for all $x\in X$; racks satisfying this are 
called \emph{quandles}. The enveloping groups of finite indecomposable quandles are characterized 
in the following lemma.

\begin{lem}\label{lem:semidirect}
Let $X$ be a finite indecomposable quandle and let $b\in X$. 
Then $G_X\simeq D_X\rtimes\langle b\rangle$ and 
\begin{align}\label{eq:semidirect}
D_X \simeq\langle \gamma_x: x\in X \mid
\gamma_{b} = e, \ \gamma_{x}\, \gamma_{b\triangleright y}=
\gamma_{x\triangleright y}\,\gamma_{b\triangleright x}, \ x,y\in X \rangle.
\end{align}
\end{lem}

\begin{proof} Let $\widetilde{D}$ be the group in the right hand side of \eqref{eq:semidirect}.
Since $X$ is a quandle, there is a unique group automorphism $\phi\colon\widetilde{D}\to\widetilde{D}$ such that $\phi(\gamma_{x})= \gamma_{b\triangleright x}$ for all $x\in X$.
Thus, there is a unique action of $\ndZ$ on $\widetilde{D}$ by automorphisms, where a generator $z$ of $\ndZ$ acts via $\phi$.
Moreover, the map 
\begin{align*}
G_X\to\widetilde{D}\rtimes_{\phi}\ndZ ,\quad 
e_x\mapsto (\gamma_{x},z),\quad 
\end{align*}
is a group isomorphism with inverse $(\gamma_{x},0)\mapsto e_xe_b^{-1}$, $(e,z)\mapsto e_b$. In order to check that the inverse map is well-defined, one uses the relations in $G_X$
\begin{align*} e_x e_b^{-1}e_be_ye_b^{-2}=e_{x\triangleright y}e_b^{-1} e_be_xe_b^{-2},\quad x,y\in X. \end{align*}
Now the lemma follows since $D_X$ is generated by the elements $e_xe_b^{-1}$, $x\in X$.
\end{proof}

In what follows, all racks are finite and can be realized as subsets
of finite groups stable under conjugation, unless explicitly stated otherwise.

\subsection{Affine racks}\label{subsec:affine}

Let $\varGamma$ be a finite abelian group. For any $\aut\in\Aut \varGamma $, 
$\varGamma$ with
the operation $(x,y)\mapsto x\triangleright y=
(\id- \aut)(x)+ \aut(y)$ for all $x,y\in \varGamma$ is a rack, denoted $\Aff(\varGamma,\aut)$.
Racks like this are called \emph{affine}.

\medbreak
We say that $\aut\in\Aut \varGamma$
is \emph{fixed-point free} if $\aut(x)=x$ implies $x=0$. 
Affine racks arising from fixed-point free automorphisms have convenient properties. For example, in this case
$\Aff(\varGamma,\aut) \simeq \prescript{\varGamma \rtimes \aut}{}{\aut}$.

\begin{pro}\label{pro:affine-derived}
Let $\aut \in\Aut \varGamma $ be fixed-point free; set $X=\Aff(\varGamma,\aut)$. 
Then:
\begin{enumerate}[leftmargin=*,,label=\rm{(\roman*)}]
\item\label{item:affine-derived-1} The rack $X$ is faithful and indecomposable.

\medbreak
\item\label{item:affine-derived-2} The derived subgroup $D_X = [G_X, G_X]$ is finite and nilpotent. 

\medbreak
\item\label{item:affine-derived-3} 
We have $D_X/[D_X, D_X]\simeq \varGamma$. In particular, 
$[D_X,D_X]= D_X\cap Z(G_X)$ and 
$D_X$ is a central extension of $\varGamma$.

\medbreak
\item\label{item:affine-derived-4} 
If $\varGamma$ is a $p$-group, where $p$ is a prime number, then $D_X$ is a $p$-group.
 \end{enumerate} 
\end{pro}

\begin{proof} We identify $X$ with $\varGamma$; as above $e: X\to G_X$ is the canonical map. 
Observe that $0 \triangleright y = \aut(y)$, hence $0 \triangleright^{-1} y = \aut^{-1}(y)$
and
\begin{align}\label{eq:affine}
x \triangleright (0 \triangleright^{-1} y) &= (\id- \aut)(x)+ y &
\text{ for all } x, y &\in \varGamma. 
\end{align}
This implies \ref{item:affine-derived-1}, see 
\cite[\S1.3.8]{MR1994219}, \cite[Remark 3.14]{MR2786171}.

We now prove \ref{item:affine-derived-2}. Lemma \ref{lem:enveloping-center}
says that $D_X$ is finite. Since $X$ is indecomposable,
$D_X$ is generated by $\gamma_{x} = e_x e_{0}^{-1}\in G_X$
for all $x \in \varGamma$, cf. Lemma \ref{lem:semidirect}. From the identity
 \eqref{eq:affine}we get that 
\begin{align*}
\gamma_{x}\gamma_{z} \rightharpoonup y &= 
\gamma_{x} \rightharpoonup \left((\id- \aut)(z)+ y\right) = (\id- \aut)(x + z)+ y, &
\text{ for all } x, z, y &\in \varGamma, 
\end{align*}
hence $[\gamma_{x},\gamma_{z}]$ acts trivially on $X$ and therefore 
is central in $G_X$ by Lemma~\ref{lem:enveloping-center}. 
Hence $[D_X,D_X]\leq Z(G_X)$ and $D_X$ is nilpotent. 

\medbreak
We next deal with \ref{item:affine-derived-3}.
To show that $D_X/[D_X,D_X]\simeq \varGamma$, we first observe that there is a 
group homomorphism $\Psi: D_X \to \varGamma$ determined by 
\begin{align*}
\gamma_{x} & \mapsto(\id- \aut)(x), & x &\in \varGamma.
\end{align*}
This follows from Lemma~\ref{lem:semidirect} with $b=0$, checking the relations in \eqref{eq:semidirect}. Clearly, $d \rightharpoonup y = \Psi(d) + y$ for any $d \in D_X$
and $y \in \varGamma$, as this holds when $d = \gamma_x$. Hence 
$\ker \Psi = Z(G_X) \cap D_X$. Since $\varGamma$ is abelian, $\Psi$ induces a 
group homomorphism 
\begin{align*}
\eta : D_X/[D_X,D_X] \to \varGamma
\end{align*}
which is surjective and has $\ker \eta  \simeq (Z(G_X) \cap D_X) / [D_X,D_X]$.
Now since $\id - \aut$ is bijective by hypothesis, the map
\begin{align*}
\theta: \varGamma &\to D_X/[D_X,D_X], & \theta((\id - \aut)(x)) &= \overline{\gamma_x},&
x &\in \varGamma,
\end{align*}
is well-defined, where $\overline{z}$ is the class of $z$ in $D_X/[D_X,D_X]$. We claim that $\theta$ is a morphism of groups.
Indeed, taking $y =0$ in the second relation of \eqref{eq:semidirect} we get
\begin{align*}
\overline{\gamma_{x}} =
\overline{\gamma_{(\id - \aut)(x)}}\,\overline{\gamma_{\aut x}}.
\end{align*}
Looking again at the second relation of \eqref{eq:semidirect}, we compute
\begin{align*}
\overline{\gamma_{(\id- \aut)(x)+ \aut(y)}} &= \overline{\gamma_{x\triangleright y}}=
\overline{\gamma_{x}}\, \overline{\gamma_{\aut(y)}}\,\overline{\gamma_{\aut (x)}^{-1}}
= \overline{\gamma_{(\id - \aut)(x)}}\,\overline{\gamma_{\aut (y)}}.
\end{align*}
Since both $\aut $ and $1-\aut $ are bijective, we see that $\overline{\gamma_{u + v}} 
= \overline{\gamma_{u}}\,\cdot \overline{\gamma_{v}}$ for all $u,v \in \varGamma$ and
the claim follows.
Now obviously $\eta \theta = \id_{\varGamma}$, while $\theta \eta = \id_{D_X/[D_X,D_X]}$
because $\theta \eta $ is a morphism of groups that fixes the generators $\overline{\gamma_x}$.
Hence $\eta $ is an isomorphism and in particular $D_X\cap Z(G_X)=[D_X,D_X]$.

Finally we prove \ref{item:affine-derived-4}. By \ref{item:affine-derived-2}, 
$D_X$, being finite and nilpotent, is isomorphic to the direct product of its Sylow subgroups. 
Hence $D_X$ is a $p$-group if and only if $D_X/[D_X,D_X]$ is a $p$-group,
and the claim follows from \ref{item:affine-derived-3}.
\end{proof}

\begin{rem}
\label{rem:affine_indec}
The converse of Proposition~\ref{pro:affine-derived}\ref{item:affine-derived-1} holds by a similar proof:
If $X=\Aff(\Gamma,\aut)$ is a finite indecomposable affine rack, then 
$\aut $ is fixed-point free. 
\end{rem}

\begin{cor}
\label{cor:abelian_centralizers}
Let $G$ be a group generated by a conjugacy 
class $X$ of $G$, $x\in X$, and let $D=[G,G]$. Assume that 
$X\simeq  \Aff(\varGamma,\aut)$
for some finite abelian group $\varGamma$ and
a fixed-point free automorphism $\aut$ of $\varGamma$. The following statements hold:
\begin{enumerate}[leftmargin=*,,label=\rm{(\roman*)}]
\item $D/[D,D]\simeq  \varGamma$. Hence
$[D,D]= D\cap Z(G)$ and 
$D$ is a central extension of $\varGamma$.

\medbreak
\item $C_G(x)=\langle x,[D,D]\rangle$. 

\medbreak
\item $C_G(x)$ is abelian. 

\medbreak
\item\label{item:abelian-cetralizers} $G$ is solvable.
\end{enumerate} 
\end{cor}

\begin{proof} Clearly, $\Inn X = \varGamma \rtimes \langle \aut \rangle$
has derived subgroup $\varGamma$. Let $\varLambda = D/[D,D]$.
The natural maps $G_X \overset{p}{\to} G \overset{q}{\to} \Inn X$ induce the vertical morphisms
in the following diagram with exact rows:
\begin{align*}
\xymatrix@R=1pc{ 0 \ar @{->}[r] & [D_X, D_X] \ar @{->}[r]\ar @{->}^{p_1}[d] 
& D_X  \ar @{->}[r] \ar @{->}^{p_2}[d] & \varGamma  \ar @{->}^{p_3}[d]
\ar @{->}[r] & 0 
\\
0 \ar @{->}[r] &[D, D] \ar @{->}[r] \ar @{->}^{q_1}[d] & D  \ar @{->}[r]
\ar @{->}^{q_2}[d] & \varLambda  \ar @{->}^{q_3}[d]  \ar @{->}[r] & 0
\\
0 \ar @{->}[r] & 0 \ar @{->}[r] & \varGamma  \ar @{->}[r] 
& \varGamma  \ar @{->}[r] & 0.}
\end{align*}
This implies that $\varLambda \simeq \varGamma$. The rest of the proof is straightforward. For the last claim, observe that $G$ is solvable because $[D,D]$, $D/ [D,D]$ and $G/D$ are abelian. 
\end{proof}

\begin{lem}
\label{lem:Aff}
Let $G$ be a finite group, $\varGamma \lhd G$ and $g\in G$. Then
\begin{align*}
X=\{\gamma g\gamma^{-1}\mid \gamma \in \varGamma \}
\end{align*}
is a subrack of $G$. 
Let $\aut \in \Aut \varGamma$, $\aut(\gamma)=g\gamma g^{-1}$ for all $\gamma \in \varGamma $.
If $\aut $ is fixed-point free, then $|X|=|\varGamma |$.
If furthermore $\varGamma$ is abelian, then
$X$ is isomorphic to $\Aff (\varGamma ,\aut)$.
\end{lem}

\begin{proof} If $\gamma_1,\gamma_2\in \varGamma$, then
\begin{align*}
(\gamma_1 \triangleright g) (\gamma_2 \triangleright g)(\gamma_1 \triangleright g)^{-1}
&=\gamma_1g\gamma_1^{-1}\gamma_2g\gamma_2^{-1}\gamma_1g^{-1}\gamma_1^{-1}
\\ &= \left(\gamma_1g\gamma_1^{-1}\gamma_2 \right) \triangleright g
= \left(\gamma_1g\gamma_1^{-1}\gamma_2 g^{-1}\right) \triangleright g.
\end{align*}
Now $\gamma_1g\gamma_1^{-1}\gamma_2 g^{-1} \in \varGamma$ by hypothesis, and $X$
is a subrack of $G$.
If $\aut$ is fixed-point free, then $|X|=|\varGamma |$ by a direct verification. 
If $\varGamma$ is abelian, then $\Aff (\varGamma ,\aut) \to X$, $\gamma \mapsto \gamma \triangleright g$ is the desired isomorphism.
\end{proof}

\subsection{Braided vector spaces of rack type}\label{subsec:bvs-rack}
Let $X$ be a rack. A \emph{2-cocycle} on $X$ is a function $q: X \times X \to \fie^{\times}$
such that 
$q_{x,y \triangleright z} q_{y,z} = q_{x \triangleright y, x \triangleright z} q_{x,z}$
for all $x,y,z \in X$. Such a 2-cocycle $q$ gives rise to a braided vector space
$(\fie X, c^q)$ with basis $(v_x)_{x \in X}$ where $c^q \in GL(\fie X \otimes \fie X)$ 
is defined by
\begin{align*}
c^q (v_x \otimes v_y) &= q_{x,y} v_{x \triangleright y} \otimes v_x,& x,y &\in X.
\end{align*}
If a rack is abelian, then any function $q: X \times X \to \fie^{\times}$ 
is a cocycle and the
corresponding braided vector space is said to be of \emph{diagonal type}. If moreover for all $x,y\in X$ there exist $a_{x,y}\in -\ndN _0$ such that 
\[
q_{x,y}q_{y,x}=q_{x,x}^{a_{x,y}},
\]
then the corresponding braided vector space 
is said to be of \emph{Cartan type}.

\medbreak
A realization of the braided vector space $(\fie X, c^q)$ over a group $G$
is a collection $(V, g, \nu)$, where $V \in \yd{\fie G}$, 
$g: X \to \supp V$ is an isomorphism of racks and $\nu: X \to V$ is a map such that 
\begin{align*}
(\nu_x)_{x\in X} &\text{ is a basis of } V, & \nu_x &\in V_{g_x}, & g_x \cdot \nu_y &= q_{xy} \nu_{x \triangleright y} & \text{ for all } x,y &\in X.
\end{align*}
Thus the braided vector space underlying $V$ is isomorphic to $(\fie X, c^q)$.
Clearly, the image of $g$ is always a union of conjugacy classes of $G$; if $X$ is indecomposable, it is necessarily a conjugacy class of $G$.
The realization is \emph{strong} if the image of $g$ is a conjugacy class of $G$ 
that generates $G$. 
We say that $(\fie X, c^q)$ is (strongly) realizable over $G$ if it admits a (strong) realization.
The same braided vector space can be realized over many groups in many ways.
The Yetter-Drinfeld modules of our interest 
are realizations of braided vector spaces as above. 

\begin{rem}
There are braided vector spaces $(\fie X, c^{\bq})$ where $X$ is a rack and $\bq: X \times X \to GL(n,\fie)$ is a non-abelian 2-cocycle, i.e., it satisfies the analogue of the condition above. We refrain of stating explicitly the definition of their realization.
\end{rem}

\subsection{Main examples}\label{subsec:main-examples}
Here we discuss the Yetter-Drinfeld modules  with indecomposable support
over solvable groups
which have finite-dimensional Nichols algebra 
and are relevant for the present paper;
see Subsection~\ref{subsec:Nichols} for the notion of Nichols algebra.
These Yetter-Drinfeld modules 
are strong realizations of the braided vector spaces 
$(\fie X, c^{q})$ listed below.
Their Nichols algebras are denoted by $\NA(X, q)$.

\medbreak
\begin{exa}\label{exa:affine} 
Consider the affine racks
$\Aff(\ndF_p, d) = \Aff(\varGamma, \aut)$ with $\varGamma = \ndF_p$, $p$ prime, and 
$\aut$ being multiplication by $d \in \ndF_p^{\times}$, where 
\begin{align}\label{eq:exa-affine}
(\ndF_p,d)\in\{(\ndF_3,2),(\ndF_5,2),(\ndF_5,3),(\ndF_7,3),(\ndF_7,5)\}.
\end{align}
Let $X$ be a rack in this list and let $\epsilon$ be the constant cocycle $-1$; 
then $(\fie X, c^{\epsilon})$ is realizable over $\Inn X\simeq \ndF_p \rtimes C_{p-1}$, where 
$C_{k}$ for $k\ge 2$ denotes the cyclic group of order $k$. 
More generally, if $(\fie X, c^{\epsilon})$ is strongly realizable over a finite group $G$, 
i.e.,  $G$ is generated by $\supp V \simeq X$, 
then $G$ is a quotient of the enveloping group $G_X$ and covers the group
$\ndF_p \rtimes C_{p-1}$. 

\medbreak
If $(\ndF_p, d)$ is as in \eqref{eq:exa-affine}, then 
$\dim \NA(\Aff(\ndF_p, d), \epsilon) = (p-1)^{p-1} p$,
cf.~\cite{MR1667680,MR1800714,MR1800709,MR1994219} 
for the case where $\fie$ is of characteristic zero and
\cite{MR2803792} for positive characteristics.
\end{exa}

\medbreak
\begin{exa}
\label{exa:T}
Consider the affine rack
$\tetra = \Aff( \ndF_4,\aut)$ 
where $\aut$ 
is multiplication by $\bomega \in \mathbb F_4\backslash \{0,1\}$ (called the tetrahedron rack).
It is also isomorphic to each of the conjugacy classes of $3$-cycles in $\mathbb A_4 \simeq \Inn \tetra$.
Let $G$ be the enveloping group of $\tetra$ and let $V\in \yd{\fie G}$ be such that $\supp V\simeq \tetra $ and $\dim V_x=1$ for all $x\in \supp V$.
Then there exist $g_1,g_2,g_3,g_4\in \supp V$, $v_i\in V_{g_i}$ for all $1\le i\le 4$, and $q\in \fie ^\times $, $\epsilon \in \{1,-1\}$, such that $g_i\cdot v_j$ for $1\le i,j\le 4$ is given by Table~\ref{tab:T}.

\begin{table}[ht]
\centering
\caption{Yetter-Drinfeld modules with support isomorphic to $\tetra$.}
\begin{tabular}{c|cccc}
$V$ & $v_{1}$ & $v_{2}$ & $v_{3}$ & $v_{4}$\tabularnewline
\hline
$g_{1}$ & $qv_{1}$ & $qv_{4}$ & $qv_{2}$ & $qv_{3}$\tabularnewline
$g_{2}$ & $qv_{3}$ & $qv_{2}$ & $\epsilon qv_{4}$ & $\epsilon qv_{1}$\tabularnewline
$g_{3}$ & $qv_{4}$ & $\epsilon qv_{1}$ & $qv_{3}$ & $\epsilon qv_{2}$\tabularnewline
$g_{4}$ & $qv_{2}$ & $\epsilon qv_{3}$ & $\epsilon qv_{1}$ & $qv_{4}$\tabularnewline
\end{tabular}
\label{tab:T}
\end{table}

There are two relevant $2$-cocycles, one corresponding to 
$(\epsilon,q)=(1,-1)$ and the other to  $(\epsilon,q)=(-1,\omega )$, where $\omega\in \fie $ such that $\omega^2+\omega+1=0$.

\medbreak
The dimensions of the Nichols algebras are as follows.
\begin{itemize}[leftmargin=*]\renewcommand{\labelitemi}{$\circ$}
 \item $\dim \NA(V) = 36$ if $(\epsilon,q)=(1,-1)$ and $\car \fie =2$, by \cite{MR2803792};
 
 \medbreak
 \item $\dim \NA(V) = 72$ if $(\epsilon,q)=(1,-1)$ and $\car \fie \ne 2$, by \cite{MR1800709,MR2803792};

 \medbreak
 \item $\dim \NA(V) = 5184$ if $(\epsilon,q)=(-1,\omega )$ and $\omega\in \fie $ such that $\omega^2+\omega+1=0$, by \cite{MR2891215}.
\end{itemize}
The corresponding braided vector spaces 
are not realizable over $\mathbb A_4$.
\end{exa}

\medbreak
\begin{exa} \label{exa:S4a}
Let $\Oc^4_2$ be the conjugacy class of transpositions in $\mathbb S_4$.
We consider:

\begin{itemize}[leftmargin=*]\renewcommand{\labelitemi}{$\circ$}
\item The constant cocycle $\epsilon = -1$;

\medbreak
\item
the 2-cocycle $\chi_4: \Oc^4_2 \times \Oc^4_2 \to \fiet$ defined by 
\begin{align*}
\chi_{4}(\sigma, (ij)) &= \begin{cases} 1, & \text{if } \sigma(i) < \sigma (j),
\\ -1, & \text{if } \sigma(i) > \sigma (j),\end{cases} &
\sigma &\in \Oc^4_2, \ i <j. 
\end{align*}
\end{itemize}

\medbreak
The braided vector spaces $(\fie \Oc^4_2, c^{\epsilon})$ and $(\fie \Oc^4_2, c^{\chi_4})$
are realizable over $\mathbb S_4$. 
Both Nichols algebras $\NA(\Oc^4_2, \epsilon)$ and $\NA(\Oc^4_2, \chi_{4})$
have dimension $576$, see \cite{MR1667680,MR1800714} for fields of characteristic zero and
\cite{MR2803792} for fields of positive characteristics.
\end{exa}

\medbreak
\begin{exa} \label{exa:S4b}
Let $\Oc^4_4$ be the conjugacy class of $4$-cycles in $\mathbb S_4$
and let $\epsilon$ be the constant cocycle $-1$; 
then $(\fie \Oc^4_4, c^{\epsilon})$ is realizable over $\mathbb S_4$.
The Nichols algebra $\NA(\Oc^4_4, \epsilon)$ also
has dimension $576$, see \cite{MR1994219,MR2803792}.
\end{exa}

\begin{rem}[\cite{MR2926571}]
\label{rem:quandles6}
Up to rack isomorphisms, $\Oc^4_2$ and $\Oc^4_4$ are the only indecomposable quandles of size $6$. 
Recall that $G_X$ is the enveloping group of a rack $X$.
\begin{enumerate}[leftmargin=*,label=\rm{(\roman*)}]
\item Let $X = \Oc^4_2$ and $g\in X$. Then the centralizer $C_{G_X}(e_g)$ in $G_X$ is isomorphic to $\ndZ\times C_2$. Indeed $C_{G_X}(e_g) = \langle e_g, e_g e_h^{-1}\rangle$, where $h$ is the unique element of $X\setminus \{g\}$ commuting with $g$. Note that $e_ge_h^{-1}\in [G_X,G_X]$  has order two.

\medbreak
\item Let $X = \Oc^4_2$  and $g\in X$. Then  
$C_{G_X}(e_g)$ is isomorphic to $\ndZ\times C_4$.
Indeed, $C_{G_X}(g)= \langle e_g, e_g e_h^{-1}\rangle$, where $h$ is the unique element of $X\setminus \{g\}$ commuting with $g$. Note that $e_g e_h^{-1}\in [G_X,G_X]$ has order four.
\end{enumerate}
\end{rem}

\begin{lem} Let $(\fie X, c^q)$ be as in one of the Examples 
\emph{\ref{exa:affine}, \ref{exa:T}, \ref{exa:S4a}} or 
\emph{\ref{exa:S4b}}. If $(\fie X, c^q)$ is realizable over
a finite group $G$, then $\vert G\vert$ is even and $\gcd(\vert G \vert, 105) \neq 1$.
\end{lem}

\begin{proof} Using Example \ref{exa:inner}, we see that
\begin{align*}
\Inn \Aff(\ndF_p, d) &\simeq \ndF_p \rtimes C_{p-1},&
\Inn \tetra &\simeq \mathbb A_4,&
\Inn \Oc^4_2 &\simeq \mathbb S_4 \simeq \Inn \Oc^4_4.
\end{align*}
Let $(V, g, \nu)$ be a realization of $(\fie X, c^q)$ over $G$.
Then $G$ acts on $X$ by $k \rightharpoonup x = y$ if 
$k \triangleright g_x = g_y$, $k \in G$, $x \in X$.
Since $g_z \triangleright g_x = g_{z \triangleright x}$ for $x, z \in X$, 
we have $\Inn X \leq \operatorname{Im} \vartheta$, where 
$\vartheta: G \to \mathbb S_X$ is induced by the action
$\rightharpoonup$. Hence $\vert\Inn X \vert$ divides $\vert G \vert$.
\end{proof}
 
\begin{rem}\label{rem:FK5}
Only two finite-dimensional Nichols algebras with indecomposable 
support over non-solvable groups are known, namely $\NA(\Oc^5_2, \epsilon)$ and $\NA(\Oc^5_2, \chi_{5})$ (here $\chi_{5}$ 
is analogous to $\chi_{4}$);
both have dimension $2^{12}3^45^2=8\,294\,400$.
\end{rem}

\section{Graded and filtered Hopf algebras}
\subsection{Graded Hopf algebras}\label{subsection:graded-bosonization}
Recall that an $\ndN_0$-graded Hopf algebra
is a Hopf algebra $A$ provided with a decomposition
$A=\bigoplus_{i \in \ndN_0} A(i)$ 
such that the structure maps are homogeneous of degree zero, that is
\begin{align*}
A(i)A(j) &\subseteq A(i+j),& \Delta(A(i)) 
& \subseteq \sum_{0\le j \le i} A(j) \otimes A(i-j),&
\Ss(A(i)) &= A(i),
\end{align*}
for every $i,j \in \ndN_0$. As is well-known, it follows that
$1\in A(0)$, $\varepsilon(A(k))=0$ for $k>0$.
For short, we shall say graded Hopf algebra as no other gradings will appear.
Graded Hopf algebras in the braided tensor category $\yd{H}$ of Yetter-Drinfeld
modules over a Hopf algebra $H$ are defined similarly.

\medbreak
Given a graded Hopf algebra $A$,
$A(0)$ is a Hopf subalgebra of $A$ and the inclusion and the homogeneous projection 
\begin{align*}
\iota: A(0)\hookrightarrow A \qquad \text{ and } \qquad \pi:A \twoheadrightarrow A(0)
\end{align*}
are Hopf algebra maps. Then $A$ is a right $A(0)$-comodule via the \emph{canonical right $A(0)$-coaction}
\begin{align}\label{eq:def-rho}
\rho=(\id \ot \pi)\Delta: A\to A\ot A(0).
\end{align}
Hence $A^{\co \rho} = \{a \in A: \rho(a) = a \otimes 1\}$ 
is a graded Hopf algebra in $\yd{A(0)}$ and
$A\simeq A^{\co \rho}\# A(0)$, the bosonization of $A^{\co \rho}$ by $A(0)$.
In some occasions, we write $A^{\co A(0)}$ instead of $A^{\co \rho}$.
See e.g. \cite[Section 4.3]{MR4164719} for details.

\subsection{Filtered Hopf algebras}
Recall that a \emph{decreasing Hopf algebra filtration} of a Hopf algebra $H$
is a family $\cF = (H_i)_{i\in \ndN_0}$ of subspaces of $H$ such that
\begin{enumerate}[leftmargin=*,,label=\rm{(\arabic*)}]
\item $H_0=H$, and $H_i\supseteq H_j$ for all $0\le i\le j$,

\medbreak
\item $H_iH_j\subseteq H_{i+j}$ for all $i,j\ge 0$, 

\medbreak
\item $\Delta(H_i)\subseteq \sum_{0 \le j \le i} H_j\ot H_{i-j}$ for all $i\ge 0$, and 

\medbreak
\item $\varepsilon (H_i)=0$ and $\Ss(H_i)\subseteq H_i$ for all $i\ge 1$.
\end{enumerate}
Given such a filtration $\cF$, the \emph{associated graded Hopf algebra} is 
\begin{align*}
\gr_\cF H &=\bigoplus_{i \in \ndN_0} \gr^i_\cF H ,& &\text{where } 
&\gr^i_\cF H &= H_i/H_{i+1}.
\end{align*}
Thus, $(\gr _\cF H)^{\co \gr _\cF^0 H}$ is a Hopf algebra in $\yd{\gr _\cF^0 H}$
and
$\gr_\cF H$ is isomorphic to the bosonization
$(\gr _\cF H)^{\co \gr _\cF^0 H} \# \gr _\cF^0 H$. 
 The canonical projections are denoted by 
\begin{align*}
\varrho_i: H_i &\to \gr^i_\cF H, & i &\in \ndN_0.
\end{align*}

\medbreak
The filtration $\cF$ is \emph{separated} if $\bigcap_{n\in \ndN_0} H_n = 0$;
in this case $\gr_\cF H \simeq H$ as vector spaces.

\begin{exa}\label{rem:powers}
For any Hopf ideal $J$ of a Hopf algebra $H$, the 
family $(J^i)_{i \in \ndN_0}$ of powers of $J$ (where $J^0 = H$)
is a decreasing Hopf algebra filtration of $H$. 
If $\dim H < \infty$ and $J$ is nilpotent, then this filtration is separated.
\end{exa}

We give now a basic result on graded Hopf algebras with a further filtration defined
by an ideal of the homogeneous component of degree 0.

\begin{pro}\label{pro:A}
Let $A=\bigoplus_{i \in \ndN_0} A(i)$ be a graded Hopf algebra,
$\rho$ the canonical right $A(0)$-coaction,
and $V = {A(1)}^{\co \rho}$.
Let $J_0$ be a Hopf ideal of the Hopf algebra $A(0)$,  and let $B = A(0) / J_0$,
\begin{align*}
J &= J_0\oplus \bigoplus_{i\geq 1} A(i),& 
 \cF &= (J^i)_{i \in \ndN_0}.
\end{align*}

\begin{enumerate}[leftmargin=*,,label=\rm{(\alph*)}]
\item\label{item:A-filtration}
$J$ is a graded Hopf ideal of $A$, $\cF$ is a decreasing Hopf algebra filtration,
and $\gr_\cF A = \bigoplus_{i \in \ndN_0} \gr^i_\cF A$ is a right $B$-comodule via the map $\overline{\rho}$ defined as 
in \eqref{eq:def-rho} with respect to the projection 
$\overline{\pi}:\gr_\cF A \twoheadrightarrow B$.
Thus 
\begin{align*} \gr_\cF A\simeq (\gr_\cF A)^{\co \overline{\rho}}\# B; \end{align*}
here $(\gr_\cF A)^{\co \overline{\rho}} = \cR = \bigoplus_{i \in \ndN_0} \cR(i)$
is a graded Hopf algebra in $\yd{B}$.

\medbreak
\item\label{item:A-gen-deg1-0} Assume that $A$ is generated by $A(0)\oplus A(1)$. Let 
\[
\ad J_0 (V) =\lspan_{\fie}\{x_{(1)}v\Ss(x_{(2)})\mid x\in J_0,v\in V\}.
\]
Then 
$J/J^2=\cR(1)B$, and 
there is an isomorphism in $\yd{B}$
\begin{align} \label{eq:prop-A-isom}
\cR(1) \simeq (J_0/J_0^2)^{\co \overline{\rho}}\, \oplus \, V/\ad J_0(V).
\end{align}

\medbreak
\item\label{item:A-gen-deg1}
Assume that $A$ is generated by $A(0)\oplus A(1)$. Then
$\gr_\cF A$ is generated as an algebra by 
$\gr_\cF^0 A \oplus \gr_\cF^1 A = B \oplus \cR(1)B$, and $\cR$ is generated as an algebra by $\cR(1)$.
\end{enumerate}
\end{pro}

\begin{proof}
\ref{item:A-filtration}: The first statement is clear because $J$ is the kernel of the composition $A \to A(0) \to B$ and the rest follows from the previous considerations.

\medbreak
\ref{item:A-gen-deg1-0}: 
Note first that $J/J^2=\cR(1)B$ by~\ref{item:A-filtration}.
Moreover, since $A$ is generated by $A(0)\oplus A(1)$, we conclude that
\begin{align}\label{eq:J^2}
J^2=J_0^2\oplus (J_0A(1)+A(1)J_0) \oplus 
\bigoplus_{i \ge 2} A(i).
\end{align}
Since the multiplication induces a bijection $V\ot A(0)\to A(1)$, we have
\begin{align*} J_0A(1)+A(1)J_0= J_0 V A(0) + V J_0. \end{align*}
Note that $xv=x_{(1)}v\Ss(x_{(2)})x_{(3)}$ for all $x\in A(0)$ and $v\in V$. Since $J_0$ is a Hopf ideal, it follows that $J_0V\subseteq \ad J_0(V)A(0)+A(1)J_0$.
Hence 
\begin{align*} J_0A(1)+A(1)J_0= \ad J_0(V) A(0) + V J_0. \end{align*}
We conclude that
\begin{align*}
J/J^2& \simeq \left(J_0\oplus V A(0)\right)/\big( J_0^2\oplus (\ad J_0 (V)) A(0) + VJ_0)\big)\\
& \simeq J_0/J_0^2\oplus VA(0)/(\ad J_0(V) A(0) + V J_0). 
\end{align*}
Now note that
$\ad J_0(V) \otimes A(0) + V \otimes J_0$ is the kernel of the canonical projection
$V\ot A(0)\to V/ \ad J_0(V) \ot B$ and that
\begin{align*} \big( V/\ad J_0(V) \ot B\big)^{\co \overline \rho}=V/\ad J_0(V) \in\yd{B}. \end{align*}
This implies~\eqref{eq:prop-A-isom}. 

\ref{item:A-gen-deg1}: By construction, 
$\gr_{\cF}A$ is generated by 
\[
A/J+J/J^2=\gr_{\cF}^0A\oplus \gr_{\cF}^1A.
\]
Hence \ref{item:A-gen-deg1} follows from \ref{item:A-gen-deg1-0}. 
\end{proof}

\subsection{Nichols algebras}\label{subsec:Nichols}

Let $H$ be a Hopf algebra, $V \in \yd{H}$, and 
$\cB = \bigoplus_{n \in \ndN_0} \cB(n)$ a graded Hopf algebra in $\yd{H}$;
assume that it is connected, i.e., $\cB(0) \simeq \fie$ and that $\cB(1) \simeq V$.
One says that 
\begin{itemize}[leftmargin=*]\renewcommand{\labelitemi}{$\circ$}
\item $\cB$ is a \emph{pre-Nichols algebra} of $V$ if it is generated as algebra by $\cB(1)$;

\medbreak
\item $\cB$ is a \emph{post-Nichols algebra} of $V$ if its subspace of primitive elements equals $\cB(1)$;

\medbreak
\item $\cB$ is a Nichols algebra of $V$ if it is both pre-Nichols and post-Nichols.
There exists a unique Nichols algebra of $V$ up to isomorphism, denoted by $\NA(V)$.
\end{itemize} 

There are alternative definitions of Nichols algebras;
here is one to be used later.
Let $M_n: \SG n \to \BG{n}$ be the Matsumoto section
of the standard group epimorphism from the braid group $\BG{n}$ to 
the symmetric group $\mathbb S_n$; 
in particular, 
 \[
 M_n(xy)=M_n(x)M_n(y)
 \]
for all $x,y\in\SG n$ such that $\ell(xy)=\ell(x)+\ell(y)$, where $\ell$ is the usual length of the Coxeter group $\SG n $. 
Define $\Omega_n \in \End \left(T^n(V)\right)$ and, in turn, the \emph{quantum symmetrizer} 
$\Omega \in \End\left(T(V)\right)$ by
\begin{align*}
\Omega_n &= \sum_{s \in \mathbb S_n} M_n(s), &\Omega &= \sum_{n \geq 2} \Omega_n.
\end{align*}
Then $\NA(V) \simeq T(V) / \ker \Omega$. This shows that 
$\NA (V)$ is uniquely determined as an algebra and 
coalgebra by the braided vector space $(V,c)$.
 
\subsection{Filtrations of Yetter-Drinfeld modules}\label{subsec:flags}
\emph{Reference: \cite[\S 3.4]{MR4298502}.}

Let $H$ be a Hopf algebra and $V \in \yd{H}$. An ascending filtration 
\[
0 =V_0 \subsetneq V_1 \dots \subsetneq V_d = V
\]
by subobjects in $\yd{H}$ of $V$
will be called a \emph{flag} of $V$; a flag is \emph{full} if the subfactors $V_{i} / V_{i-1}$ are simple.
The category of increasingly filtered objects is a braided tensor one, 
hence filtered Hopf algebras in $\yd{H}$ make sense.
Given a filtered $V \in \yd{H}$, its associated graded vector space 
$\gr V = \bigoplus_{i} V_{i} / V_{i-1}$ is a graded object in $\yd{H}$.
So, if $\cB$ is a filtered Hopf algebra in $\yd{H}$, then $\gr \cB$ is a 
graded Hopf algebra in $\yd{H}$. Also, if $V \in \yd{H}$ is filtered,
then any pre-Nichols algebra of $V$ in $\yd{H}$ is a filtered graded Hopf algebra.

\medbreak
The following result follows easily from \cite[Lemma 3.7]{MR4298502}.

\begin{lem}\label{lem:filtered}
Let $0 =V_0 \subsetneq V_1 \dots \subsetneq V_d = V\in \yd{H}$ 
be a flag of finite-dimensional Yetter-Drinfeld modules and let $\cB$ be
a pre-Nichols algebra of $V$. Then $\gr \cB$ is a
pre-Nichols algebra of $\gr V$ and $\dim \NA(\gr V) \leq \dim \NA(V)$. \qed
\end{lem}

The following two results are valid in any characteristic.

\begin{pro}
\label{pro:consequence}
Let $G$ be a non-abelian group and $V, W \in \yd{\fie G}$ 
absolutely simple. Assume that $\supp(V\oplus W)$ generates $G$ and 
$c_{W,V}c_{V,W}\ne\id$.
If, in addition,
$\dim\NA(V\oplus W)<\infty$, then $\min\{\dim V,\dim W\}\leq 2$ and 
$\max\{\dim V,\dim W\}\leq 4$. 
\end{pro}

\begin{proof}
The claim follows by inspection of \cite[Theorem 2.1 and Table 1]{MR3656477}.
\end{proof}

\begin{cor}\label{cor:absolutely_irreducible}
Let $G$ be a non-abelian group and $V \in \yd{\fie G}$. Assume that
$\supp V$ is a conjugacy class of $G$ that generates $G$
and $\dim \NA(V)<\infty $. 
Then $V$ is absolutely simple.
\end{cor}

\begin{proof}
Let $X=\supp V$.  Since $X$ is a conjugacy class of $G$ which generates $G$ and since $G$ is non-abelian, 
the rack $X$ is indecomposable and $|X| \geq 3$.
Moreover, $|X|\le \dim V\le \dim \NA (V)<\infty $.
Given a finite field extension $E \vert \fie $, the composition length of the Yetter-Drinfeld module $E\ot_\fie  V$ is bounded by $\dim_\fie  V$.
Let $L \vert \fie $ be a finite field extension such that
the composition length $r$ of $L\ot_\fie  V$ is maximal.
Assume that $V$ is not absolutely simple. Then $r\ge 2$.
Let
\begin{equation}
\label{eq:CS}
0=Y_0\subsetneq Y_1\subsetneq \cdots \subsetneq Y_r=L\ot _\fie  V \end{equation}
be a composition series of $L\ot_\fie  V$. Since $X$ is indecomposable, we conclude that $\supp Y_i/Y_{i-1} =X$ for all $1\le i\le r$. The composition series \eqref{eq:CS} defines a filtration on $L\ot_\fie  V$. Let $W=\gr (L\ot_\fie  V)$ be the associated graded Yetter-Drinfeld module.
Then
\begin{align*} W=\bigoplus_{i=1}^r Y_i/Y_{i-1} \end{align*}
is a direct sum of $r$ absolutely simple Yetter-Drinfeld modules over $LG$, each of them having support $X$. In particular, $W$ is braid-indecomposable. By Lemma~\ref{lem:filtered},
\begin{align*} \dim \NA (W)\le \dim \NA (V)<\infty, \end{align*}
contradicting 
to Proposition~\ref{pro:consequence}. 
Therefore, $V$ is absolutely simple.
\end{proof}

Next, we discuss an application of Proposition~\ref{pro:psistar} and Corollary~\ref{cor:absolutely_irreducible}.

\begin{pro}
\label{cor:phcenter}
Let $G$ be a non-abelian group and $V \in \yd{\fie G}$. Let
$N \lhd G$ be contained in the kernel of the representation of $G$ on $V$; let $\pi:G\to G/N$ be the canonical map. Then $V$ and $\pi_*(V)$ have the same braiding. Moreover, if $\supp V$ is a conjugacy class of $G$ generating $G$ and if $\dim\NA (V)<\infty$, 
then $\pi_*(V)$ is absolutely simple.
\end{pro}

\begin{proof}
Since $N$ acts trivially on $V$, the Yetter-Drinfeld modules $V$ and
\begin{align*} \pi_*(V)= \fie (G/N)\ot_{\fie G} V \end{align*}
have the same braiding. In particular, $\dim \NA (\pi_*(V))=\dim \NA (V)$.

Assume that $\supp V$ is a conjugacy class of $G$ generating $G$ and that $\NA (V)$ is finite-dimensional. Then $\dim \NA (\pi_*(V))<\infty $ and $\supp \pi_*(V)$ is a conjugacy class of $G/N$ generating $G/N$. Moreover, $G/N$ is non-abelian by Lemma~\ref{lem:quotient_YD}.
Thus, by Corollary~\ref{cor:absolutely_irreducible}, $\pi_*(V)$ is absolutely simple.
\end{proof}

\section{Nichols algebras in prime characteristic}\label{sec:prime-char}
Throughout this section, $G$ denotes a finite group, $p$ is a prime number
and the base field $\fie$ has characteristic $p$, except in Lemma \ref{lem:CurtisReiner69.9}. We collect preliminaries on Nichols algebras under these assumptions.

\subsection{Nichols algebras of decomposable modules}\label{subsec:nichols-decomposable}
Here we summarize results from \cite{MR3605018,MR3625122,MR4099895,MR3313687} 
and state a few consequences thereof.

\begin{thm}\label{pro:diagtypewithcoinvariants}
Let $I$ be a finite set with a decomposition $I=I_1\cup I_2$ into two non-empty disjoint subsets.
Let $\vbd$ be a finite-dimensional braided vector space of diagonal type with braiding matrix 
$\bq = (q_{ij})_{i,j\in I}\in (\fiet)^{I\times I}$; let 
$(x_i)_{i\in I}$ be a basis of $\vbd$ such that
\begin{align*}
c (x_i\ot x_j) &= q_{ij}x_j\ot x_i& &\text{for all }i,j\in I.
\end{align*}
Assume that
\begin{enumerate}[leftmargin=*,label=\rm{(A\arabic*)}]
\item\label{A1} $q_{ij}=1$ for all $i,j\in I_1$,

\medbreak
\item\label{A2} $q_{ij}q_{ji}=q_{ik}q_{ki} \ne 1$ for all $i\in I_1$, $j,k\in I_2$, and

\medbreak
\item\label{A3} the polynomial $\prod_{i\in I_1} (t-q_{ij}q_{ji})$ is in
$\ndF_p[t]$ for all $j\in I_2$.
\end{enumerate}

\medbreak
\noindent Then $\dim\NA (\vbd) < \infty$ if and only if the Dynkin diagram of $\bq$ belongs to the following list:
\begin{enumerate}[leftmargin=*,label=\rm{(\Alph*)}]
\item\label{item:exaA} \Dchaintwo{$-1$}{$-1$}{$1$} where
$p =3$; 

\medbreak
\item\label{item:exaB} \Dchaintwo{$-1$}{$\zeta$}{$1$}
where $p =5$ and $\zeta\in\{2,3\}\subseteq\ndF_5$;

\medbreak
\item\label{item:exaC} 
\Dchaintwo{$-1$}{$\zeta$}{$1$}
where $p =7$ and $\zeta\in\{3,5\}\subseteq\ndF_7$;

\medbreak
\item\label{item:exaD} 
\Dchainthree{}{$1$}{$\zeta $}{$a$}{$\zeta ^{-1}$}{$1$} where
$p=2$, $\zeta ^2+\zeta+1=0$ and $a\in \{1,\zeta,\zeta^{-1}\}$. 
 \end{enumerate}
\end{thm}

Note that the braiding matrices of the examples in \ref{item:exaA}, \ref{item:exaB} and 
\ref{item:exaC} can be characterized by $|I_1|=|I_2|=1$, $p\in \{3,5,7\}$, $q_{ii}=1$, $q_{jj}=-1$, and $q_{ij}q_{ji}$ is a primitive root of the residue class ring $\ndZ/p\ndZ$ for $i\in I_1$, $j\in I_2$.

\begin{proof}
The Nichols algebras corresponding to the Dynkin diagrams in the list 
have finite dimension
by \cite[Theorem 5.1]{MR3313687} and \cite[Corollary 3.3]{MR3625122};
see row $6^{\prime\prime\prime}$ in Table 5.2, 
row $15^{\prime}$ of Table~5.3, row $18$ of Table~5.4 in 
\cite{MR3313687};
and rows $8$ and $15$ of Table 1 in \cite{MR3625122}. 

\medbreak
Suppose now that $\dim\NA(\vbd)<\infty$.

\medbreak
Assume first that $|I_1|=|I_2|=1$.

\medbreak
If $p=2$, 
then $q_{ij}q_{ji}\not\in\{0,1\}$ for $i\in I_1$, $j\in I_2$ by \ref{A2}, contradicting \ref{A3}.

\medbreak

If $p=3$, then Theorem~5.1 and Table~5.2 of \cite{MR3313687} 
imply that $\bq$ has Dynkin diagram
\begin{center} 
\Dchaintwo{$-1$}{$-1$}{$1$}.
\end{center}

If $p=5$, then Theorem~5.1 and Table~5.3 of \cite{MR3313687} 
imply that $\bq$ has Dynkin diagram
\begin{center} 
\Dchaintwo{$-1$}{$\zeta $}{$1$} \qquad where $\zeta \in \{2,3\}\subseteq \ndF_5$.
\end{center}

If $p=7$, then Theorem 5.1 and Table 5.4 of \cite{MR3313687} 
imply that $\bq$ has Dynkin diagram
\begin{center} 
\Dchaintwo{$-1$}{$\zeta $}{$1$}\qquad
where $\zeta \in \{3,5\}\subseteq \ndF_7$.
\end{center}

\medbreak
According to Theorem~5.1 and Table~5.5 of \cite{MR3313687}, if $p>7$
then
there are no braided vector spaces of diagonal type satisfying assumptions \ref{A1}, \ref{A2} and \ref{A3} and
having a finite-dimensional Nichols algebra.

\medbreak
Assume now that $|I_1|+|I_2|=3$. If $p=2$ and $|I_1|=2$,
then Corollary~3.3 and Table~1 of \cite{MR3625122} together with assumptions \ref{A1}, \ref{A2} and \ref{A3}
imply that the braiding matrix $\bq$ has Dynkin diagram
\begin{center} 
\Dchainthree{}{$1$}{$\zeta $}{$a$}{$\zeta ^{-1}$}{$1$} 
\end{center}
with $\zeta ^2+\zeta+1=0$ and $a\in \{1,\zeta ,\zeta^{-1}\}$.
On the other hand, if $p=2$ and $|I_1|=1$, then $q_{ij}q_{ji}=1$ for all $i\in I_1$, $j\in I_2$ because of \ref{A3}, contradicting \ref{A2}.

\medbreak
According to Corollary 3.3 and Tables 2 and 3 of \cite{MR3313687}, if $|I_1|+|I_2|=3$ and $p>2$
then
there are no braided vector spaces of diagonal type satisfying assumptions \ref{A1}, \ref{A2} and \ref{A3} and
having a finite-dimensional Nichols algebra.

\medbreak
Assume now that $|I_1|+|I_2|\ge 4$ and $\dim\NA(V)<\infty$.
Let $W$ be a braided subspace of $V$ of dimension $4$ corresponding to a subset of $I$ containing at least one element both from $I_1$ and from $I_2$.
Then $\NA (W)$ is finite-dimensional. However, 
by Corollary~3.2 and Table~1 of \cite{MR4099895} at least one of \ref{A1} and \ref{A2} is not true, resulting in a contradiction.
Hence \ref{A1} and \ref{A2} cannot hold simultaneously for $V$.
\end{proof}

We record the following well-known fact for further use.

\begin{lem}\label{lem:CurtisReiner69.9} \cite[Corollary (69.9)]{MR2215618}
Let $G$ be a finite group, let $\Fie$ be any field, let $M$ be a completely reducible $\Fie 
G$-module and let $L\vert \Fie$ be a finite separable field extension. 
Then $L\ot _{\Fie} M$ is a completely reducible $L G$-module. \qed
\end{lem}

\begin{cor}
\label{cor:examples_abeliantype}
Let $G$ be a finite abelian group and
$V, W \in \yd{\fie G}$. Assume that 
\begin{enumerate}[leftmargin=*,label=\rm{(\alph*)}]
\item\label{item:examples_abeliantype} $V$ is absolutely simple and 
$\supp V$ generates $G$;

\medbreak
\item 
$\supp W = \{e\}$, the trivial  $\fie G$-module is not a submodule of $W$, and 
there exists a simple $\ndF_p G$-module $\widehat{W}$  such that
\begin{align*} W\simeq  \fie \ot_{\ndF_p}\widehat{W}. \end{align*}
\end{enumerate}
Then $G$ is cyclic. Moreover, 
$\dim\NA(V\oplus W)<\infty$ 
if and only if either of the following holds:
\begin{itemize}[leftmargin=*]\renewcommand{\labelitemi}{$\circ$}
\item $\dim V=\dim W=1$ and $V\oplus W$ is of diagonal type
as in conditions \ref{item:exaA}, \ref{item:exaB} or \ref{item:exaC} 
in  Theorem~\ref{pro:diagtypewithcoinvariants}, where 
the vertex with label $1$ corresponds to $W$;

\medbreak
\item $\dim V=1$, $\dim W=2$ 
and there exists a finite field extension $L \vert \fie$
such that $L \otimes_{\fie} (V \oplus W)$ is of diagonal type
as in condition \ref{item:exaD} 
in  Theorem~\ref{pro:diagtypewithcoinvariants}, where 
the vertices with label $1$ correspond to basis vectors of $W$.
\end{itemize}
\end{cor}

\begin{proof}
Since $G$ is abelian, hypothesis \ref{item:examples_abeliantype} implies
that $|\supp V|=1$ and $G$ is cyclic with generator $g$.
Since $\ndF_p$ is a perfect field and $\widehat{W}$ is a semisimple $\ndF_p G$-module,  we know by Lemma~\ref{lem:CurtisReiner69.9}
that there exists a finite field extension $L \vert \fie$ such that 
\[
W^L = L\otimes_{\ndF_p}\widehat{W}\simeq  L\ot_{\fie} W
\]
is a direct sum of absolutely simple $LG$-modules. 
Since the trivial $\fie G$-module is not a submodule of $W$, 
the trivial $LG$-module is not a submodule of $W^L$.
The assumptions on $V$ imply that $V^L = L\ot_{\fie} V$ (with the extended $L$-linear action and coaction of $LG$ on it) is an absolutely simple Yetter-Drinfeld module over $LG$.
By assumption, the Nichols algebra
\[ 
\NA (W^L\oplus V^L)\simeq  L\ot_{\fie} \NA (W\oplus V) 
\]
is finite-dimensional.
Since $G$ is abelian and $W^L\oplus V^L$ is a direct sum of absolutely simple objects in $\yd{LG}$, we conclude that $W^L\oplus V^L$ is of 
diagonal type. 

\medbreak
We now claim that Theorem~\ref{pro:diagtypewithcoinvariants} applies in this setting, where $(x_i)_{i\in I_1}$ is a basis of $W^L$ of eigenvectors of the action of the generator $g$ of $G$, and  $(x_k)_{k\in I_2}$ is a basis of the one-dimensional Yetter-Drinfeld module $V$. Indeed, \ref{A1} holds because
$\supp W=\{e\}$, \ref{A2} holds because $W$ contains 
no trivial $\fie G$-submodule, and \ref{A3} holds since 
$\prod_{i\in I_1} (t-q_{ij}q_{ji})\in L[t]$ is the characteristic polynomial of the action of $g$ on $W$, and hence is in $\ndF_p[t]$. Then the corollary 
follows 
from Theorem~\ref{pro:diagtypewithcoinvariants}.
\end{proof}

We say that a family $(M_1,\dots,M_\theta)$ of Yetter-Drinfeld modules over $\fie G$ is \emph{braid-indecomposable} if for each decomposition $\{1,2,\dots,\theta\}=I_1\cup I_2$ with $I_1,I_2\ne \emptyset$ and $I_1\cap I_2=\emptyset $ there exist $i\in I_1$, $j\in I_2$ such that
the braiding satisfies
\begin{align*}
c_{M_j,M_i}c_{M_i,M_j} \ne  \id_{M_i \otimes M_j}.
\end{align*}

\begin{thm}\label{pro:nonabtypewithcoinvariants}
Let $G$ be a non-abelian group. 
Let $\theta\geq3$ and $M=(M_1,\dots,M_\theta)$ be a braid-indecomposable 
tuple of absolutely simple Yetter-Drinfeld modules over $\fie G$ such that
$\supp(M_1\oplus\cdots\oplus M_{\theta})$ generates $G$ 
and $\supp M_i=\{e\}$ for some $i$. Then $\dim\NA(M_1\oplus\cdots\oplus M_{\theta})=\infty$. 
\end{thm}

\begin{proof}
This follows from the Classification Theorem in \cite[p. 301]{MR3605018}, by
the definition of skeletons of finite type, 
\cite[Figure 2.1]{MR3605018} and since $\supp M_i=\{e\}$ for some $i$. 
\end{proof}

The next two results collect useful information on Nichols algebras of 
a direct sum of two Yetter-Drinfeld modules, 
one of them having trivial support.
For these claims we introduce the group
\begin{align}\label{eq:def-G}
\mathbb{G} = \langle g,\epsilon:g\epsilon=\epsilon^{-1}g,\epsilon^3=e\rangle.
\end{align} 

\begin{pro}
\label{pro:rank2}
Let $G$ be a non-abelian group and $V$ and $W$ be absolutely simple Yetter-Drinfeld
modules over $\fie G$. Assume
that 
\begin{align}\label{eq:hypothesis-rank2}
\supp W &= \{e\}, & \supp V &\text{ generates } G,& c_{W,V}c_{V,W} &\ne\id.
\end{align}
Let $g \in \supp V$. Then $\dim\NA(V\oplus W)<\infty$ if and only if 
$p=2$, 
$G$ is an epimorphic image of $\mathbb{G}$, $C_G(g)=\langle g\rangle$, and one of the
following conditions hold: 
\begin{enumerate}[leftmargin=*,,label=\rm{(\alph*)}]
\item\label{item:rank2-1} There exist degree-one representations $\chi$ of $C_G(g)$ 
and $\sigma$ of $G$ such that 
\begin{align*} 
& V \simeq M(g,\chi), &
& W \simeq M(e,\sigma), \\
&(\chi(g)-1)(\chi(g)\sigma(g)-1) =0, & 
&(3)_{\sigma(g)} = 0.
\end{align*}

\medbreak
\item\label{item:rank2-2} There exist a degree-one representation 
$\chi$ of $C_G(g)$ and 
an absolutely irreducible representation $\sigma$
of $G$ of degree two such that 
\begin{align}\label{eq:rank2-b}
V &\simeq M(g,\chi), &
W &\simeq M(e,\sigma), &
\chi(g)=\sigma(g^2) &= 1, &
\sigma(e+\epsilon+\epsilon^2) &= 0.
\end{align}
\end{enumerate}
\end{pro}

\begin{proof}
The assumptions imply that $V\oplus W$ is isomorphic to
one of the
examples of \cite[Section 1]{MR3656477}. Examples 1.2, 1.3 and
1.5, 1.9 therein have no direct summand with central support. In Example 1.7
the cocycle condition is not fulfilled, since the support of $W$ 
is trivial. In Example~1.10, the conditions
of the proposition are met if and only if
$V$ and $W$ satisfy the conditions in the first item. 
In Example~1.11 the conditions
of the proposition are fulfilled if and only if
$V$ and $W$ satisfy the conditions in the second item.
\end{proof}

\begin{cor}
\label{cor:rank2}
Let $G$, $V$ and $W$ be as in Proposition \ref{pro:rank2}. Assume
that \eqref{eq:hypothesis-rank2} holds and, moreover, that 
\begin{align} \label{eq:rep-defined-prime-field}
W\simeq\fie \ot_{\ndF_p}\cW
\end{align}
for an $\ndF_pG$-module $\cW$.
Let $g \in \supp V$. Then $\dim\NA(V\oplus W)<\infty$ if and only if 
$p=2$, 
$G$ is an epimorphic image of $\mathbb{G}$, $C_G(g)=\langle g\rangle$, and
there exist a degree-one representation 
$\chi$ of $C_G(g)=\langle g\rangle$ and 
an absolutely irreducible representation $\sigma$
of $G$ of degree two such that 
\eqref{eq:rank2-b} holds.
\end{cor}

\begin{proof}
By Proposition \ref{pro:rank2}, $p = 2$. If $\dim W=1$, then the assumption \eqref{eq:rep-defined-prime-field} says that $\sigma$ is the trivial representation of $G$,
which contradicts the hypothesis $c_{W,V}c_{V,W} \ne\id$. Therefore
$V$ and $W$ have to be as in Proposition~\ref{pro:rank2} \ref{item:rank2-2}.
\end{proof}

\begin{rem} \label{rem:V+W}
More precisely, the Yetter-Drinfeld modules $V$ and $W$ in Corollary~\ref{cor:rank2} have the following constraints. 
Equation $\sigma (g^2)=\id $ implies that $1$ is the only eigenvalue of $\sigma(g)$, since $\car \fie =2$. On the other hand, $\sigma (g)\ne \id $
since otherwise $\sigma(\epsilon^2)=\sigma(g\epsilon g^{-1})=\sigma(\epsilon)$ and hence
$0=\sigma(e+\epsilon+\epsilon^2)=\sigma(e)$, a contradiction.
\end{rem}

We shall use the following well-known facts.

\begin{rem}\label{rem:YD_charp}
Let $G$ be a group. Then any $\fie G$-module $X$ becomes 
a Yetter-Drinfeld module over $\fie G$ with trivial coaction 
$x\mapsto 1 \ot x$ for all $x\in X$.
\end{rem}
 
\begin{pro}\label{pro:sumoftwoveryspecialYD}
Let $G$ be a finite group and consider the following data:
\begin{itemize}[leftmargin=*]\renewcommand{\labelitemi}{$\circ$}
\item $Y\in\yd{\fie G}$ such that $\supp Y$ 
is a conjugacy class of $G$ which generates $G$,

\medbreak
\item $\cX \neq 0$ is a semisimple $\ndF_p G$-module without $G$-invariant direct summands; 
set 
\begin{align*}
X = \cX^{\fie} = \fie\ot_{\ndF_p} \cX\in \yd{\fie G}
\end{align*}
with trivial coaction as in Remark~\ref{rem:YD_charp}. 
\end{itemize}

\medbreak
\noindent Assume that $\dim\NA(X \oplus Y)<\infty$.
Then $\cX$ is a simple $\ndF_p G$-module and $Y$ is absolutely simple in $\yd{\fie G}$.
\end{pro}

\begin{proof}
We split the proof into two steps.

\begin{step}\label{step:one}
We prove the claim first under the assumption that $Y$ is semisimple
and all its simple summands are absolutely simple.
\end{step}

Let $E$ be a (finite) splitting field for the group algebra $\ndF_p G$;
$E \vert \ndF_p$ is separable because the field $\ndF_p$ is perfect. 
Let $a\in E$ be such that $E=\ndF_p[a]$. Then $L=\fie[a]$ is separable over $\fie $.
Since $X$ is semisimple, Lemma~\ref{lem:CurtisReiner69.9} says that
\begin{align*}
X^L=L\otimes_{\ndF_p}\cX \simeq L\ot_{\fie} \cX^{\fie}
\end{align*}
is semisimple; by the choice of $L$, 
any simple summand is an absolutely simple $L G$-module. 
Since $\cX$ has no non-zero $G$-invariant elements, $X^L$ also has no non-zero $G$-invariant elements. Let $Y^L = L\ot_\fie Y$. The assumptions on $Y$ on Step~1 imply that $Y^L$
is a direct sum of absolutely simple Yetter-Drinfeld modules over $L G$.
Thus $X^L\oplus Y^L$ is a direct sum of absolutely simple objects in $\yd{L G}$.
By hypothesis, the Nichols algebra
\begin{align*} \NA (X^L\oplus Y^L)\simeq L\ot_\fie \NA (X \oplus Y) \end{align*}
is finite-dimensional. Take $v$ in a simple component of $X^L$ and 
$w$ in a simple component of $Y^L$, $w$ homogeneous of degree $g$. If $v\ot w\ne 0$, then
\begin{align*}
c^2 (v \otimes w) &= c(w \otimes v) = g \cdot v \otimes w \neq v \otimes w, 
\end{align*}
since $\cX$ has no $G$-invariants and $\supp Y$ generates $G$. 
This implies that $X^L\oplus Y^L$ is braid-indecomposable.

\medbreak
If $G$ is non-abelian, then Theorem~\ref{pro:nonabtypewithcoinvariants} implies
that $X^L$and $Y^L$ are absolutely simple in $\yd{L G}$, hence 
$Y$ is absolutely simple in $\yd{\fie G}$ by the assumption of the Step and
$X$ is a simple $\ndF_p G$-module.

\medbreak
Assume now that $G$ is abelian. Then $\supp Y = \{h\}$ for some $h \in H$, being a
conjugacy class of $G$.
Moreover, $X^L\oplus Y^L$ is a direct sum of 
absolutely simple Yetter-Drinfeld modules over $L G$, hence it is of diagonal type. 
Let $(x_i)_{i\in I_1}$ be a basis of $X^L$ and 
$(x_j)_{j\in I_2}$ be a basis of $Y^L$ 
such that $c(x_i\ot x_j)\in Lx_j\ot x_i$ for all $i,j\in I$, where $I=I_1\cup I_2$. 
Since the $L G$-coaction on $X^L$ is trivial, 
$X^L$ does not contain any non-zero $G$-invariant element, 
and $G$ is generated by $h$, the assumptions in Theorem~\ref{pro:diagtypewithcoinvariants} are fulfilled. 
Then Theorem~\ref{pro:diagtypewithcoinvariants} 
implies that $X$ is a simple $\ndF_p$-module and $Y$ is absolutely simple 
in $\yd{L G}$ (since it is one-dimensional). Indeed, \ref{item:exaD}
is discarded because $\zeta \neq \zeta^{-1}$.

\begin{step}
The general case. Let
\end{step} 
\begin{align*} 0\subseteq Y_1\subseteq Y_2\subseteq \cdots \subseteq Y_r=Y \end{align*}
be a full flag of Yetter-Drinfeld submodules of $Y$. 
Let $L$ be a (finite) splitting field for the group algebra $\fie G$,
hence also for the Drinfeld double $D(\fie G)$,
and set $Y_{i}^L = L\ot_\fie Y_i$. 
Then $Y_{i+1}^L /Y_i^L$ is a direct sum of absolutely simple 
objects in $\yd{L G}$ and 
\begin{align*} 0\subseteq Y_1^L \subseteq Y_2^L \subseteq \cdots \subseteq Y_r^L
\subseteq X^L\oplus Y^L \end{align*}
is a flag of Yetter-Drinfeld submodules of $X^L \oplus Y^L$. 
The associated graded Yetter-Drinfeld module $\gr (X^L \oplus Y^L) \simeq X^L \oplus \gr Y^L$ 
satisfies the assumptions in Step \ref{step:one};
hence $X$ is a simple $\ndF_p G$-module and $\gr(Y^L)$ is absolutely simple in $\yd{L G}$. 
Thus $Y^L$ is absolutely simple in $\yd{L G}$ and $Y$ is absolutely simple in $\yd{\fie G}$.
\end{proof}

\subsection{Filtrations associated to a normal {\it p}-subgroup}\label{subsec:reduction}

We start by a consequence of Proposition \ref{pro:A}. Recall that $p=\car\fie >0$.
Consider the following setting:

\begin{itemize}[leftmargin=*]\renewcommand{\labelitemi}{$\circ$}
\item $G$ is a finite group, $\varGamma \lhd G$ is a normal $p$-subgroup; 

\medbreak
\item $V\in \yd{\fie G}$, $A = \NA (V)\# \fie G = \bigoplus_{n \in \ndN_0} A(n)$
where $A(n) = \NA^n(V)\# \fie G$;

\medbreak
\item $J_{\varGamma}$ is the augmentation ideal of $\fie \varGamma$
and $J$ is the left ideal $A (J_\varGamma + V)$ of $A$.
\end{itemize}

Let $\overline{\pi}:\gr_\cF A \twoheadrightarrow \fie (G/\varGamma)$ 
be the canonical surjection and $\overline{\rho} = (\id \ot \overline{\pi})\Delta$,
cf. \eqref{eq:def-rho}.
\begin{pro}\label{cor:gamma_2} 
Consider the filtration
$\cF=(J^i)_{i\ge0}$. 
\begin{enumerate}[leftmargin=*,,label=\rm{(\alph*)}]
\item\label{item:graded-normalp} 
The ideal $J$ is a Hopf ideal, the graded ring associated to $\cF$  is
\begin{align*} 
\gr_\cF A \simeq(\gr_\cF A)^{\co \overline\rho }\# \fie (G/\varGamma), 
\end{align*}
and $(\gr_\cF A)^{\co \overline \rho }$ is generated as an algebra by
\begin{align*} 
(\gr _\cF A)^{\co \overline \rho }(1)\simeq J_\varGamma /J_\varGamma ^2
\oplus V/ \ad J_\varGamma (V). 
\end{align*}

\medbreak
\item\label{item:graded-normal-sep} If $\dim A < \infty$, then
the filtration $\cF$ is separated and $\dim A = \dim \gr _\cF A$.
\end{enumerate}
\end{pro}

\begin{proof} The left ideal $J_0 = \fie G J_{\varGamma}$ 
is a Hopf ideal of $\fie G$, being the kernel of the projection $\fie G \to \fie (G/\varGamma)$.
Now $A$ is generated in degree 0 and 1, hence
\begin{align*}J = J_0 \oplus \bigoplus_{i\geq 1} A(i).\end{align*}
Then \ref{item:graded-normalp} follows from Proposition~\ref{pro:A};
for \eqref{eq:prop-A-isom}, since $V$ is stable by the adjoint action, one has that 
\begin{align*}
\ad J_0(V) = \ad J_\varGamma \left(\ad \fie G (V)\right) = \ad J_\varGamma (V).
\end{align*}

\medbreak
\ref{item:graded-normal-sep}: Since $\dim A < \infty$, it follows that $G$ is finite and 
$\bigoplus_{i\geq m} A(i) = 0$ for some $m \in \ndN_0$.
Since $\varGamma$ is a normal $p$-subgroup of $G$
and $\car \fie =p$, the filtration
$(J_{\varGamma}^n)_{n \in \ndN_0}$ of $\fie \varGamma$ is separated, hence so is
the filtration $(J_{0}^n)_{n \in \ndN_0}$ of $\fie G$. 
The remaining claim is then clear.
\end{proof}

Recall that the \emph{Frattini subgroup} of the finite group $\varGamma$ is defined 
as the intersection $\Phi(\varGamma)$ of 
the maximal subgroups of $\varGamma$.

\begin{lem}
\label{lem:quotient_JGamma} 
Let $\Phi=\Phi(\varGamma)$. The following statements hold: 

\begin{enumerate}[leftmargin=*,label=\rm{(\alph*)}]
\item\label{item:quotient_JGamma-1}
The multiplication of $\varGamma$ induces an $\ndF_p$-vector space structure on $\varGamma/\Phi$.

\medbreak
\item\label{item:quotient_JGamma-2} 
$\varGamma/\Phi$ is an $\ndF_pG$-module via $g\cdot (x\Phi)=gxg^{-1}\Phi$
for all $g\in G$, $x\in \varGamma $.

\medbreak
\item\label{item:quotient_JGamma-3} 
$J_\varGamma $ is a $\fie G$-module via conjugation and $J_\varGamma^2$
is a $\fie G$-submodule of $J_\varGamma$.

\medbreak
\item\label{item:quotient_JGamma-4}
 The map $\kappa\colon \varGamma/\Phi \ot_{\ndF_p} \fie\to J_\varGamma /J_\varGamma^2$
given by 
\begin{align*}
x\Phi \otimes t &\mapsto t(x-1)+J_\varGamma^2,& x \in \varGamma, \ t \in \fie,
\end{align*} 
is an isomorphism of $\fie G$-modules. 
\end{enumerate}
\end{lem}

\begin{proof}
\ref{item:quotient_JGamma-1} Since $\Gamma$ is a finite $p$-group, $\varGamma/\Phi$ is an elementary abelian $p$-group, hence a vector space over $\ndF_p$. 

\medbreak 
\ref{item:quotient_JGamma-2} and \ref{item:quotient_JGamma-3} follow by direct calculations. 

\medbreak
\ref{item:quotient_JGamma-4} 
The map $\kappa$ is linear, since for any $x,y\in \varGamma$, 
\begin{align*}
xy-1+J_\varGamma^2 &=(x-1)+(y-1)+(x-1)(y-1) +J_\varGamma^2
=(x-1)+(y-1)+J_\varGamma^2. 
\end{align*}
Recall that $\Phi$ is generated by $p$-powers and commutators of elements of $\varGamma$ (see \cite[Lemma~4.5]{MR2426855}). For each such generator $y$ of $\Phi$, $y-1\in J_\varGamma^2$ by the beginning of the proof of~\ref{item:quotient_JGamma-4}.
Hence $\kappa$ is a well-defined $\fie$-linear surjection, which clearly
commutes with the action of $G$. 
To prove that $\kappa $ is invertible,
note that $J_\varGamma^2$ is the subspace of $J_\varGamma$ spanned by all elements of the form
\begin{align*} (x-1)(y-1)=(xy-1)-(x-1)-(y-1), \quad x,y\in \varGamma. \end{align*}
Thus the $\fie$-linear map $J_\varGamma \to \varGamma/\Phi \ot_{\ndF_p}\fie$ given by
\begin{align*}
x-1 &\mapsto x\Phi \ot 1, &x &\in \varGamma,
\end{align*}
annihilates $J_\varGamma^2$; the induced map $J_\varGamma /J_\varGamma^2\to
\varGamma/\Phi \ot_{\ndF_p}\fie$ is the inverse of $\kappa$.
\end{proof}

\section{A classification in positive characteristic}\label{sec:prime-char-solvable}

In this section, $p$ is a prime number
and $\car\fie = p$. 
We characterize the finite-dimensional
Nichols algebras  of a class of
Yetter-Drinfeld modules over  groups. 

\begin{thm}
\label{thm:char_p}
Let $G$ be a finite group such that 
$G/Z^*(G)$ has a non-trivial normal $p$-subgroup. 
Let  $0 \neq V \in \yd{\fie G}$. Assume that 
\begin{itemize}[leftmargin=*]\renewcommand{\labelitemi}{$\circ$}
\item 
$\supp V$ is a conjugacy class of $G$ generating $G$, and

\medbreak
\item $\dim\NA(V)<\infty$.
\end{itemize}
Then $V$ is absolutely simple 
and isomorphic to one of the Yetter-Drinfeld modules
of Examples~\emph{\ref{exa:affine}} (with $p=\dim V$), \emph{\ref{exa:T}, \ref{exa:S4a}} or \emph{\ref{exa:S4b}} (with $p=2$). 
\end{thm}

Evidently, a group $G$ satisfying the hypothesis in Theorem~\ref{thm:char_p} could not be nilpotent,
otherwise $G/Z^*(G)$ would be trivial.

\begin{proof}
By hypothesis, $O_p(G/Z^*(G))$ is not trivial. Then $O_p(G)$ is not trivial by Lemma~\ref{lem:hypercenter} \ref{item:hypercenter3}. 
Hence $G$ has a minimal normal $p$-subgroup $\varGamma$; clearly, $\varGamma \leq O_p(G)$. By Lemma~\ref{lem:minimal_normal}, $\varGamma$ is elementary abelian
(we denote it additively).
The group $G$ acts by conjugation on $\varGamma$; the space of invariants for this action is $\varGamma\cap Z(G)$. 
By the minimality of $\varGamma$, $\varGamma$ is a simple $\ndF_p G$-module via this action. 

\medbreak
\emph{Assume first that $O_p(G)\cap Z^*(G)$ is trivial.}

\medbreak
By this assumption, the space of
$G$-invariant vectors of the $\ndF_p G$-module $\varGamma$ is $0$.
Let $J_\varGamma$ be the augmentation ideal of $\fie \varGamma$. 
As in Proposition~\ref{cor:gamma_2},
we consider
\begin{align*}
A &=\NA (V)\#\fie G,&
J &= A (V+J_\varGamma),&
\cF &=(J^i)_{i\ge 0}.
\end{align*}
Since $\dim \NA (V)<\infty $ and $G$ is finite,
$\dim \gr_\cF A < \infty$. By Proposition~\ref{cor:gamma_2},
\[ 
\gr_\cF A \simeq  (\gr_\cF A)^{\co \overline\rho }\# \fie (G/\varGamma),
\]
where $\overline\rho$ is defined in Proposition~\ref{pro:A}, 
and $(\gr_\cF A)^{\co \overline \rho }$ is generated as an algebra by
\[ 
(\gr _\cF A)^{\co \overline \rho }(1)\simeq  J_\varGamma /J_\varGamma ^2\oplus V/\ad J_\varGamma (V).
\]
By Proposition~\ref{pro:A} with $J_0=A(0)J_\varGamma$,  $(\gr _\cF A)^{\co \overline \rho }(1)$ is a Yetter-Drinfeld module over \[(\gr_\cF A)(0)=A(0)/J_0 \simeq \fie(G/\varGamma ),\]
where the last isomorphism holds by Lemma~\ref{lem:J}.
Therefore there is a surjective homomorphism of braided Hopf algebras 
\[ 
(\gr _\cF A)^{\co \overline \rho }\to\NA\left(J_\varGamma /J_\varGamma ^2\oplus V/\ad J_\varGamma (V)\right).
\]
In particular, $\NA\left(J_\varGamma /J_\varGamma ^2\oplus V/\ad J_\varGamma (V)\right)$ is finite-dimensional.
Since $\varGamma$ is elementary abelian,
Lemma~\ref{lem:quotient_JGamma} \ref{item:quotient_JGamma-4} implies that there is an isomorphism
\[ J_\varGamma /J_\varGamma ^2\oplus V/\ad J_\varGamma (V)\to (\varGamma \ot _{\ndF_p} \fie)\oplus \big(V/\ad J_\varGamma (V)\big) \]
of objects in $\yd{A(0)/J_0} = \yd{\fie(G/\varGamma )}$.  Thus, $\NA\left((\varGamma \ot _{\ndF_p} \fie)\oplus (V/\ad J_\varGamma (V))\right)$ is finite-dimensional.

\medbreak
Now the Yetter-Drinfeld module $(\varGamma \ot _{\ndF_p} \fie)\oplus (V/\ad J_\varGamma (V))$ over $\fie(G/\varGamma )$
satisfies the assumptions of Proposition~\ref{pro:sumoftwoveryspecialYD} with the group
$G/\varGamma$, the Yetter-Drinfeld module $Y = V/\ad J_\varGamma (V)$ and the $\ndF_p(G/\varGamma)$-module $\cX =\varGamma$.
Proposition~\ref{pro:sumoftwoveryspecialYD} tells then that

\begin{itemize}[leftmargin=*]\renewcommand{\labelitemi}{$\circ$}

\item  $\varGamma $ is a simple $\ndF_p (G/\varGamma) $-module and 

\medbreak
\item $V/\ad J_\varGamma (V)$ is an absolutely simple 
Yetter-Drinfeld module over $\fie(G/\varGamma)$.
\end{itemize}

\begin{case}
Assume that $O_p(G)\cap Z^*(G)$ is trivial and that 
$G/\varGamma$ is abelian.
\end{case}
By the discussion above,   
$(\varGamma \ot _{\ndF_p} \fie)\oplus \big(V/\ad J_\varGamma (V)\big)$
is as in Corollary~\ref{cor:examples_abeliantype};  therefore $\dim V/\ad J_\varGamma (V)=1$, $\varGamma$ is elementary abelian of
order 3, 4, 5 or 7 and correspondingly $p= 3, 2, 5$ or 7. Recall that $\varGamma  \cap Z(G) = \{e\}$.
Let $g\in \supp V$. Then 
\[
\supp V/\ad J_\varGamma (V) = \{ g\varGamma \},
\]
and the square of the braiding of
$(\varGamma \ot _{\ndF_p} \fie)\oplus \big(V/\ad J_\varGamma (V)\big)$ satisfies the equation
\[ c^2(\gamma \ot v)=g\gamma g^{-1} \ot v
\]
for all $\gamma \in \varGamma $ and $v\in V/\ad J_\varGamma (V)$.
Since $V/\ad J_\varGamma (V)$ is absolutely simple, 
Proposition~\ref{pro:psistar} implies that $\varGamma$ 
 acts transitively on $\supp V$. 
 Let $\aut $ be the conjugation action of $g$ on $\varGamma$. By Lemma~\ref{lem:Aff}, 
we conclude that $\supp V$ is isomorphic as a rack to 
$\Aff(\varGamma ,\aut )$. Let $D=[G,G]$.
By Proposition~\ref{pro:affine-derived} \ref{item:affine-derived-4} and  Corollary~\ref{cor:abelian_centralizers},
\[ [D,D]=D\cap Z(G)\subseteq O_p(G)\cap Z^*(G)=\{e\}, \]
where the last equation holds by assumption.
By Corollary~\ref{cor:abelian_centralizers}, $C_G(g)$ is cyclic. Thus $V$ can be presented as a braided vector space by the rack
$X = \Aff(\varGamma ,\aut )$
and the constant cocycle $q$,
where $gv=qv$ for $v\in V_g$. 
By the constraints in the Cases \ref{item:exaA}, \ref{item:exaB}, \ref{item:exaC}
and \ref{item:exaD} in 
Theorem \ref{pro:diagtypewithcoinvariants}, $V$ is isomorphic to one of the 
Yetter-Drinfeld modules in Examples~\ref{exa:affine} and \ref{exa:T}.

\begin{case}
Assume that $O_p(G)\cap Z^*(G)$ is trivial and that $G/\varGamma$ is non-abelian.
\end{case}

Let $X=\supp V$ and $g\in X$; by assumption
$X = \prescript{G}{}{g}$. 
By the discussion above, $(\varGamma \ot _{\ndF_p} \fie)\oplus \big(V/\ad J_\varGamma (V)\big)$ is as in Corollary \ref{cor:rank2}.
Thus Corollary~\ref{cor:rank2} implies that $\dim V/\ad J_\varGamma (V)=3$ with support of size three,
and $\varGamma$ is elementary abelian of
order 4. Moreover, $G/\varGamma$ is an epimorphic image of $\mbG$. Let $\epsilon \in G\setminus \varGamma $
such that $g\epsilon \in X$. Then, by Corollary~\ref{cor:rank2}, there are a degree one representation $\rho$ of $C_{G/\varGamma}(g\varGamma) =\langle g\varGamma \rangle $ and an absolutely irreducible representation $\sigma $ of
$G/\varGamma $ of degree two such that
\[ V/\ad J_\varGamma (V)\simeq  M(g\varGamma ,\rho),\quad \fie\ot_{\ndF_2} \varGamma \simeq  M (\varGamma ,\sigma), \]
and
\[ \rho (g\varGamma ) =1,\quad \sigma( (g\varGamma )^2) = \id ,\quad
\sigma(\varGamma +\epsilon \varGamma +\epsilon ^2\varGamma )=0. \]
Note that the action of $g$ on  $\varGamma $ is uniquely determined by $\sigma (g\varGamma )$. By Remark~\ref{rem:V+W}, $1$ is the only eigenvalue of $\sigma(g\varGamma )$, and $\sigma (g\varGamma)\ne \id$. Hence there is a unique $\gamma \in \varGamma \setminus \{e\}$ with $g\gamma =\gamma g$. Particularly, the $\varGamma $-orbit of $g$ under conjugation has a stabilizer of order $2$, and hence an orbit length $4/2=2$. Since $H_0(\varGamma,V)=V/\ad J_\varGamma (V)$ decomposes into three $1$-dimensional homogeneous components with respect to the coaction of $G/\varGamma$, Proposition~\ref{pro:psistar} says that $|X|=2\cdot 3=6$. Since the rack $X$ is indecomposable of size 6,  
Remark \ref{rem:quandles6} says that 
it is isomorphic either to the conjugacy class $\Oc^4_2$ 
of transpositions in $\SG 4$ or 
to the conjugacy class $\Oc^4_4$ of 4-cycles in~$\SG 4$. 

\smallbreak
Assume that $X$ is isomorphic to $\Oc^4_2$. By Remark~\ref{rem:quandles6}, the  centralizer $C_G(g)$ is generated by $g$ and an element of order at most two. Since $\car\fie =2$, $V$ is absolutely simple, and $\rho(g\varGamma )=1$, this implies that $V\simeq  M(g,\tau)$ with the trivial representation $\tau $ of $C_G(g)$, yielding that $V$ appears in Example~\ref{exa:S4a}.

\smallbreak
Assume that $X$ is isomorphic to $\Oc^4_4$. By Remark~\ref{rem:quandles6}, the  centralizer of $g\in X$ in $G$ is generated by $g$ and an element $h$ with $h^4=1$. Since $\car\fie =2$, $V$ is absolutely simple, and $\rho(g\varGamma )=1$, this implies that $V\simeq  M(g,\tau)$ with the trivial representation $\tau $ of $C_G(g)$, yielding that $V$ appears in Example~\ref{exa:S4b}.

\begin{case}
Assume that $N = O_p(G)\cap Z^*(G)$ is non-trivial.  
\end{case}

Let $\varphi\colon G\to G/N$ be the canonical map. 
By Corollary~\ref{cor:absolutely_irreducible}, $V$ is absolutely simple. Then Lemma~\ref{lem:intersection} implies that $N$ acts trivially on $V$. By Proposition~\ref{cor:phcenter}, $\varphi_*(V)$ is an absolutely simple Yetter-Drinfeld module over $\fie(G/N)$ with the same braiding as the one of $V$. Hence
\[ \dim \NA (\varphi_*(V))=\dim \NA (V)<\infty . \]
By Lemma~\ref{lem:Op(G)}, $Z^*(G/N)\cap O_p(G/N)$ is trivial. 
In particular, $\dim (\varphi_*(V))_{gN}=1$ for all $g\in \supp V$ by the  previous cases.
Then $\supp V$ and $\supp \varphi_*(V)$ are isomorphic as racks by Lemma~\ref{lem:easylowerstar}.
Hence the claim follows from the the  previous cases.
\end{proof}

\section{Nichols algebras over solvable groups in 
characteristic 0}\label{sec:-char0}
In this section $\fie$ denotes a field of characteristic $0$, 
unless explicitly stated. 
The main result of the section is Theorem~\ref{thm:char_zero} on 
the classification of the
Nichols algebras of Yetter-Drinfeld modules over finite
groups whose solvable radical is different from the hypercenter,
such that the support of the module generates the group.

\subsection{Hurwitz orbits}
Let $\Oc $ be a subrack of a group $G$, e.g., a union of
conjugacy classes. For all
$n\ge 2$, the braid group $\BG{n}$ with standard generators $t_1, \dots, t_{n-1}$
acts on the $n$-fold direct product $\Oc^n$ via the Hurwitz action, concretely 
given by
\begin{align*}
t_i \cdot (x_1, \dots, x_n) &= (x_1, \dots, x_{i-1}, x_i \triangleright x_{i+1}, 
x_{i}, x_{i+2}, \dots, x_n), & 1 \le i \le n-1.
\end{align*}
The orbits of this action are known as \emph{Hurwitz orbits}. We write $|\hurw|$ to denote the cardinal of a Hurwitz orbit $\hurw$. 
The multiplication map
\[ \Oc^n\to G,\quad
(x_1,x_2,\dots,x_n)\mapsto
x_1x_2\dots x_n, \]
commutes with the Hurwitz action, and therefore is constant on Hurwitz orbits.
We call the image of this multiplication map the \emph{$G$-degree} of the Hurwitz orbit. 
The \emph{support} of a Hurwitz orbit $\hurw$ is defined as
\[ \supp \hurw=\bigcup _{(x_1,x_2,\dots,x_n)\in \hurw}
\{ x_1,x_2,\dots,x_n\}. \]

Let  $G$ be a finite group, and $V \in \yd{\fie G}$, $V \neq 0$.
For each Hurwitz orbit $\hurw\subseteq \Oc^n$, 
the Yetter-Drinfeld submodule of the tensor algebra $T(V)$ given by 
\[ T_\hurw(V)=\sum _{(x_1,x_2,\dots,x_n)\in \hurw}V_{x_1}\ot V_{x_2}\ot \cdots \ot V_{x_n} \]
is invariant under the action of the quantum symmetrizer.
We write
\[
\NA ^n_\hurw(V)=T_\hurw(V)/(\ker \Delta_{1^n}\cap T_\hurw(V))
\]
for any Hurwitz orbit $\hurw$ in $\Oc^n$. Here $\Delta_{1^n}: T^n(V) \to T^n(V)$
is the $(1, \dots, 1)$-th component of the iterated comultiplication 
$\Delta^{n-1}$, see \cite[1.3.12]{MR4164719}.
Then
\[ \NA ^n(V)=\bigoplus_{\hurw\subseteq \Oc^n} \NA ^n_\hurw(V). \]

The following lemma, valid for any field $\fie$, is a variation of \cite[Lemma 2.2]{MR2811166}.

\begin{lem}\label{lem:Hurwitz} 
Let $\hurw \subseteq \Oc^2$ be a Hurwitz orbit of $\Oc =\supp V$, 
let $g$ be the $G$-degree of $\hurw$, and let 
$\delta\colon T_\hurw(V)\to T_\hurw(V)$ 
be the restriction of $\Delta_{1^2}$ to $T_\hurw(V)$.
Assume that $|\hurw|$ is odd and that $\dim V_x=1$ for all $x\in \supp \hurw$.
\begin{enumerate}[leftmargin=*,label=\rm{(\alph*)}]

\item The rank of $\delta$ is at least $|\hurw|-1$.

\medbreak  
\item $(1+g^{|\hurw|}x^{-|\hurw|})v=(\det \delta)v$ for all $v\in V_x$, $x\in \supp \hurw$.
 \end{enumerate}
 \end{lem}

\begin{proof}
Let $g_1\in \supp \hurw$, $g_2=gg_1^{-1}$, and let $k\ge 0$ be such that $|\hurw|=2k+1$. Then $(g_2,g_1)\in \hurw$, and there exist uniquely determined elements $g_i\in \supp \hurw$ for $3\le i\le 2k+1$ such that
\[ \hurw=\{ (g_{i+1},g_i) \mid 1\le i\le 2k+1 \}, \quad g_{i+1}g_i=g, \]
where $g_{2k+2}=g_1$.
Note that then
\begin{gather*}
g_{i+2} = g_{i+1}g_ig_{i+1}^{-1}=gg_ig^{-1}
\shortintertext{for all $1\le i\le 2k$, and}
g_2 = g_1g_{2k+1}g_1^{-1} =gg_{2k+1}g^{-1}=g^{k+1}g_1g^{-(k+1)}.
\end{gather*}
Let $v_1\in V_{g_1}$ be a non-zero vector.
For all $1\le i\le k$ let 
\begin{align*}
v_{2i+1}=g^iv_1 &\in V_{g_{2i+1}},& 
v_{2i}=g^{k+i}v_1 &\in V_{g_{2i}}.
\end{align*}
Then 
there exists $\lambda \in \fie^{\times} $ such that $g_iv_i=\lambda v_i$ for all $1\le i\le 2k+1$.
We obtain from $g_{i+1}=gg_i^{-1}$ that
\begin{align*} \delta (v_{i+1}\ot v_i)
&=v_{i+1}\ot v_i+g_{i+1}v_i\ot v_{i+1}\\
&=v_{i+1}\ot v_i+\lambda^{-1}v_{i+2}\ot v_{i+1}
\end{align*}
for all $1\le i\le 2k-1$. Moreover,
\begin{align*} \delta (v_{2k+1}\ot v_{2k})
&=v_{2k+1}\ot v_{2k}+g_{2k+1}v_{2k}\ot v_{2k+1}\\
&=v_{2k+1}\ot v_{2k}+\lambda^{-1}g^{2k+1}v_1\ot v_{2k+1},\\
\delta (v_1\ot v_{2k+1})
&=v_1\ot v_{2k+1}+g_1v_{2k+1}\ot v_1\\
&=v_1\ot v_{2k+1}+\lambda^{-1}v_2\ot v_1.
\end{align*}
Now $g^{2k+1}v_1 = \nu v_1$ for some $\nu \in \fiet$,
hence $\det \delta = \nu \lambda^{- |\hurw|} + 1$; the claim follows.
\end{proof}

\subsection{Orders of braided vector spaces} \label{subsec:orders}
We use the following terminology:
\begin{itemize} [leftmargin=*]\renewcommand{\labelitemi}{$\diamond$}
\item If $R$ is a subring of $\fie$ and $\fm$ is a maximal ideal  
of $R$, then we set $\ndF = R/ \fm$.

\medbreak
\item For a subring $R$ of $\fie$, an \emph{$R$-order}  
$V_R$ of a braided vector space $(V, c)$ is a projective $R$-submodule
of $V$ satisfying $\fie \otimes_R V_R \simeq V$ and  such that 
\begin{align*}
c (V_R \otimes_R V_R) &\subseteq  V_R \otimes_R V_R, &
c^{-1}(V_R \otimes_R V_R) &\subseteq  V_R \otimes_R V_R.
\end{align*}
We write $c_R \in \Aut (V_R \otimes_R V_R)$ to denote the map induced by $c$, and 
$\NA(V_R)$ for the corresponding Nichols algebra as in
\cite[Lemma 3.5]{MR4729697}.

\medbreak
\item  For a finite group $G$, a subring $R$ of $\fie$, and
$V \in \yd{\fie G}$, an \emph{$R$-order} 
$V_R$ of $V$ is a projective $R$-submodule
of $V$ satisfying $\fie \otimes_R V_R \simeq V$, $G \cdot V_R = V_R$
and $V_R = \bigoplus_{g \in G} \left(V_R \cap V_g\right)$.
\end{itemize}

\begin{rem}
\label{rem:realization}
  In the definition of an $R$-order of a Yetter-Drinfeld module $V$, the projectivity assumption 
  on $V_R$ can be replaced by projectivity of the $R$-modules $V_R\cap V_g$, $g\in G$.
  In particular, if $V$ is simple, $R$ is a subring of $\fie$ which is 
  a Dedekind domain, and 
  there is an $R$-submodule $W$ of $V_g$ for some $g\in\supp V$ such that
  $\operatorname{rank}W=\dim V_g$ and $C_G(g)\cdot W\subseteq W$,  
  then $RG\ot _{RC_G(g)} W$ is an $R$-order of~$V$.   
\end{rem}

\begin{lem}[{\cite[Lemma 3.5]{MR4729697}}]
\label{lem:orders}
The braiding $c$ induces a structure of braided vector space on $V_{\ndF}=\ndF\ot _R V_R$ and $\ndF \ot _R \NA (V_R)$ is a pre-Nichols algebra of $V_{\ndF}$. 
If $\NA (V)$ is finite-dimensional, then 
so are $\ndF \ot _R \NA (V_R)$ and $\NA (V_{\ndF})$. \qed
\end{lem}

\noindent 
Following Takeuchi, a subspace $U$ of a braided vector space $(W,c)$ is \emph{categorical}  if
\begin{align*}
c(U\ot W) &= W\ot U & &\text{and} & c(W\ot U) &= U\ot W.
\end{align*}

Here is a very useful criterion to identify infinite-dimensional Nichols algebras; it was implicit in the proof of \cite[Theorem 7.6]{MR1886004}. 

\begin{lem}[{\cite[Lemma 3.6]{MR4729697}}]
\label{lem:specialprimitives}
Let $W\subseteq \bigoplus_{n\ge 2}\ndF \ot_R \NA^n (V_R)$
be a categorical subspace consisting of primitive elements. 
If the Nichols algebra of $V_{\ndF}\oplus W$ is infinite-dimensional, 
then $\NA (V)$ is infinite-dimensional. \qed
\end{lem}

\begin{lem}\label{lem:concrete_new_trick}
Let $R\subseteq \fie$ be a Dedekind domain,  
$G$ be a finite group and $V \in \yd{\fie G}$ admitting an $R$-order $V_R$. 
Let $\hurw$ be a Hurwitz orbit of $(\supp V)^2$ of odd length and of 
$G$-degree $g$. 
Assume that 
\begin{itemize} [leftmargin=*]\renewcommand{\labelitemi}{$\circ$}
\item $\supp V$ is a conjugacy class of $G$ generating $G$, 

\medbreak
\item $|\supp V|\ge 3$ and  $\dim V_x=1$ for all $x\in \supp V$,

\medbreak
\item $|\prescript{G}{}{g}|\ge 3$, and 

\medbreak
\item  that there are a maximal ideal $\fm $ of $R$ and an $x\in \supp \hurw$ such that
 \begin{equation}
 \label{eq:concrete_new_trick}
(1+g^{|\hurw|}x^{-|\hurw|})v_x\in \fm V_x\setminus \{0\}.
 \end{equation}
\end{itemize}
\noindent Then $\NA (V)$ is infinite-dimensional. 
\end{lem}

\begin{proof}
Note that $\NA^2(V_R)=(V_R\ot_R V_R)/\ker \Delta_{1^2}$, hence $\NA^2(V_R)$ is a torsion-free $R$-module. Since $R$ is a Dedekind domain, $\NA^2(V_R)$
is projective, and therefore
$\NA^2_\hurw(V_R)$ is projective.
Thus
\[ \dim  \ndF \ot_R\NA ^2_\hurw(V_R)=\dim \fie\ot_R \NA ^2_\hurw(V_R)=\dim \NA ^2_\hurw(V). \]
Let
\[ \overline{W}_\hurw=\ker \left(\Delta_{1^2}\colon \ndF\ot _R \NA^2_\hurw(V_R)\to \ndF \ot _R V_R\ot _R V_R\right). \]
Equation~\eqref{eq:concrete_new_trick} and Lemma~\ref{lem:Hurwitz} imply that $\dim \overline{W}_\hurw=1$. 
Let
\[ \mathrm{Stab}_G(\hurw)=\{ x\in G\mid \forall (h_1,h_2)\in \hurw:(xh_1x^{-1},xh_2x^{-1})\in \hurw \}. \]
Since the action of $G$ on $G\times G$ commutes with $\Delta_{1^2}$, $\overline{W}_\hurw$ is an $\ndF\mathrm{Stab}_G(\hurw)$-submodule of $\mathrm{Res}^G_{\mathrm{Stab}_G(\hurw)}\Big(\ndF \ot_R \NA^2(V_R)\Big)$. Let $\overline{W}=\ndF G\overline{W}_\hurw$. Then 
\[ \dim \overline{W}=|G|/|\mathrm{Stab}_G(\hurw)|\ge |G|/|C_G(g)|=|\prescript{G}{}{g}|\ge 3,\]
where the first inequality follows from $\mathrm{Stab}_G(\hurw)\subseteq C_G(g)$, and the second one from the assumption.
Since $\overline{W}_\hurw\subseteq \ndF\ot_R \NA^2(V_R)_g$, $\overline{W}$ is a Yetter-Drinfeld submodule of $\ndF\ot_R \NA^2(V_R)$ consisting of primitive elements. Moreover, $\supp \overline{W}=\prescript{G}{}{g}$. Let $W$ be a simple Yetter-Drinfeld submodule of $\overline{W}$. Then $\supp W=\supp \overline{W}$ is of size at least three. In particular, $\supp V$ and $\supp W$ do not commute. 
Since $\supp V_R$ generates $G$ and $|\supp V_R|\ge 3$,  Proposition~\ref{pro:consequence} tells that 
  \[
  \dim\NA( (\ndF \ot _R V_R)\oplus W)=\infty.
  \]
  Then 
  Lemma~\ref{lem:specialprimitives} applies, and hence $\NA (V)$ is infinite-dimensional.  
\end{proof}

\subsection{On the cocycles}\label{subsec:cocycles-examples}

Here we narrow down the possible cocycles attached to the racks in
Subsection~\ref{subsec:main-examples} that could give rise to finite-dimensional Nichols algebras.

\begin{lem} 
\label{lem:T} 
Let $G$ be a group as in Example~\ref{exa:T}, and let
$V=\fie  \tetra$ with basis $\{v_1,v_2,v_3,v_4\}$
and with $\fie G$-module structure given by Table~\ref{tab:T}, where $\epsilon \in \{1,-1\}$ and $q\ne 0$.
Assume that there exist $\zeta,r\in \fiet$ such that $q=\zeta r$, $\zeta^3=1$,
the order of $r$ is a power of two, and $\epsilon r\ne -1$. 
Then $\dim\NA(V)=\infty$.
\end{lem}

\begin{proof}
 Let $x_1,x_2,x_3,x_4\in G$ be such that $\delta(v_i)=x_i\otimes v_i$ for all $1\le i\le 4$.  Let $f\colon \supp V\to \prescript{\AG 4}{}{(123)}$ be a rack isomorphism with $f(x_1)=(123)$ and $f(x_2)=(134)$. The map $f$ induces
 a surjective group homomorphism 
 $f\colon G\to \AG 4$. 
 Let us consider the Hurwitz orbit 
 \[
 \{(x_1,x_2),(x_4,x_1),(x_2,x_4)\}\subseteq (\supp V)^2
 \]
 of $G$-degree $g=x_1x_2$ and length three. Let $R=\ndZ[\zeta, r]$ and $\fm$ be a maximal ideal of $R$ containing $2$. 
 Then $\sum_{i=1}^4 Rv_i$ is an $R$-order of $V$.  
 Since the conjugacy class of $f(g)=(234)$ in $\AG 4$ has four elements, we conclude that $|\prescript{G}{}{g}|\ge 4$.

 The rack structure of $\supp V$ implies that
 \begin{align*} g^3&=(x_1x_2)^3=x_1x_2x_1x_2x_1x_2\\
 &=x_1x_1x_3x_2x_1x_2
 =x_1^2x_1x_3x_1x_2
 =x_1^3x_1x_4x_2
 =x_1^4x_2x_3. 
 \end{align*}
 Hence
 \[ (1+g^3x_1^{-3})v_1=(1+q^3\epsilon )v_1
 =(1+(r\epsilon)^3)v_1.\]
 Note that $1-r\epsilon+(r\epsilon)^2\ne 0$, and hence $1+(r\epsilon)^3\ne 0$ since $r\epsilon\ne -1$ by assumption. Therefore Lemma~\ref{lem:concrete_new_trick} applies and $\dim \NA (V)=\infty $.
\end{proof}

\begin{lem}
\label{lem:Tcubic} 
 Let $p$ be an odd prime, let $G$ be a group as in Example~\ref{exa:T}, and let
$V=\fie  \tetra$ with basis $\{v_1,v_2,v_3,v_4\}$
and with $\fie G$-module structure given by Table~\ref{tab:T}, where $-q$ has order $p^m$ for some $m\geq1$. Then $\dim\NA(V)=\infty$.
\end{lem}

\begin{proof}
Let $x_1,x_2,x_3,x_4\in G$ be such that $\delta(v_i)=x_i\otimes v_i$,  
$1\le i\le 4$. Let $f\colon \supp V\to \prescript{\AG 4}{}{(123)}$ be a rack isomorphism with $f(x_1)=(123)$. The map $f$ induces
 a surjective group homomorphism 
 $f\colon G\to \AG 4$. 
 Consider the Hurwitz orbit
 $\{(x_1,x_1)\}\subseteq (\supp V)^2$ of $G$-degree $h=x_1^2$ and length one. 
 Since the conjugacy class of $f(h)=(132)$ in $\AG 4$ has
size $4$, we get that $|\prescript{G}{}{h}|\geq 4$.
By Lemma~\ref{lem:concrete_new_trick} with $R=\ndZ[q]$, a maximal ideal $\fm$ of $R$ containing $p$, 
and the $R$-order $\sum_{i=1}^4 Rv_i$ of $V$, it follows that  
 $\dim \NA (V)=\infty $ since $1+q\ne 0$.
\end{proof}

\begin{lem}
\label{lem:4cyclesS4}
 Let $\Oc_4^4$ be the conjugacy class of 4-cycles in $\SG 4$ and $G$ be a group having a conjugacy class $X$ 
 isomorphic as a rack to $\Oc_4^4$ and generating $G$.
 Let $x\in X$ and let $V=M(x,\rho)$ with an irreducible representation $\rho$ of $C_G(x)$.
 Assume that $\rho(x)$ is a root of one of order $2^m$ for some $m\geq0$. 
 Then $\dim\NA(V)<\infty$ if and only if $\rho(z)=-1$ for all $z\in C_G(x)\cap X$.
\end{lem}

\begin{proof}
 Let $x_1,\dots,x_6\in X$ be such that $f\colon X\to \Oc_4^4$, 
 \begin{align*}
 &x_1\mapsto (1234),\quad 
 x_2\mapsto (1243),\quad 
 x_3\mapsto (1324),\\
 &x_4\mapsto (1342),\quad 
 x_5\mapsto (1432),\quad 
 x_6\mapsto (1423),
 \end{align*}
 is a rack isomorphism. We may assume that $x=x_1$. Then $C_G(x_1)\cap X=\{x_1,x_5\}$. 
 Moreover, $f$ induces an epimorphism $f\colon G\to \SG 4$.
 The rack structure of $X$ implies the following relations in $G$:
 \begin{align*} x_1x_3&=x_2x_1=x_3x_2, &
 x_2x_4&=x_4x_2,
 \\
 x_1x_4&=x_3x_1=x_4x_3, &
 x_1x_5&=x_5x_1,
 \\
 x_3x_4&=x_4x_5=x_5x_3, &
 x_3x_6&=x_6x_3.
 \end{align*}
 
 By Remark~\ref{rem:quandles6}, $C_G(x_1)=\langle x_1,x_5\rangle$. Moreover, $(x_5x_1^{-1})^4=1$.
 Thus $\dim V_{x_1}=1$. Let  $q=\rho(x_1)$ and $\epsilon=\rho(x_5x_1^{-1})$. Then $\epsilon^4=1$ and
 \[ x_1v=qv,\quad x_5v=\epsilon qv, \]
 where $v$ is a basis of $V_{x_1}$.
 
 If $m=0$, then $\fie[v]\subseteq \NA (V)$ and hence $\dim \NA (V)=\infty$.

 \smallbreak
 Assume now that $m\geq1$. The above relations in $G$ imply that
 \[ 
 \hurw=\{ (x_2,x_1), (x_3,x_2), (x_1,x_3) \} 
 \]
 is a Hurwitz orbit of $G$-degree $g=x_2x_1$. Since the conjugacy class of \[
 f(g)=f(x_2)f(x_1)=(142)
 \]
 in $\SG 4$ has size eight, we obtain that $|\prescript{G}{}{g}|\ge 8$.
Now note that
 \[ x_3x_2x_4=x_1x_3x_4=x_1x_4x_5=x_3x_1x_5, \]
 and hence $x_2x_4=x_1x_5$ in $G$. Therefore 
 \[ g^3=(x_2x_1)^3
 =x_1x_3x_2x_1x_2x_1
 =x_1x_1x_3x_1x_2x_1
 =x_1^2x_1x_4x_2x_1=x_1^4x_5x_1. \]
 We obtain that
 \[ (1+g^{|\hurw|}x_1^{-|\hurw|})v=(1+x_1^2x_5)v=(1+\epsilon q^3)v=(1-(-\epsilon^{-1}q)^3)v. \]
 Let $R=\ndZ[q,\epsilon]$ and $\fm$ be a maximal ideal of $R$ containing two. Then 
 $\sum_{h\in G}Rhv$ is an $R$-order of $V$ and 
 \[
 (1+g^{|\hurw|}x_1^{-|\hurw|})v\in \fm V_{x_1}.
 \]
 Since $1-\epsilon^{-1}q+\epsilon^{-2}q^2\ne 0$, it also follows that
 \[
 (1+g^{|\hurw|}x_1^{-|\hurw|})v=0 \quad \Leftrightarrow \quad
 q=-\epsilon.
 \]
Hence Lemma~\ref{lem:concrete_new_trick} implies that if $q\ne -\epsilon $ then $\dim \NA (V)=\infty $.

Consider now the Hurwitz orbit
$\{(x_1,x_1)\}$ of $G$-degree $h=x_1^2$ and length one. 
Since the conjugacy class of $f(h)=(13)(24)$ in $\SG 4$ has
size $3$, we get that $|\prescript{G}{}{h}|\geq 3$. 
By Lemma~\ref{lem:concrete_new_trick} with the above maximal ideal $\fm$, if $1+q\ne 0$ then $\dim \NA (V)=\infty $.

By the last two paragraphs, if $q\ne -1$ or $q\ne -\epsilon$, then $\dim \NA (V)=\infty $. Conversely, if $q=-1$ and $q=-\epsilon$, then $\epsilon=1$ and $\dim \NA (V)<\infty $ by Example~\ref{exa:S4b}.
\end{proof}

\subsection{Nichols algebras over groups whose solvable radical is different from the hypercenter}
Throughout this subsection,  $G$ denotes a finite group.  
Recall 
that the \emph{Fitting subgroup} $F(G)$ of $G$ is the 
largest normal nilpotent subgroup of $G$ (see Subsection~\ref{subsec:Fitting}). 
The \emph{solvable radical} $\operatorname{Rad}(G)$ of $G$ is
the (unique) largest solvable normal subgroup of $G$. 
 
\begin{rem}\label{rem:Fitting}
\begin{enumerate}[leftmargin=*,label=\rm{(\alph*)}]
 
\item For any solvable  $N \lhd G$, $\Rad(G/N)=\Rad(G)/N$.

\medbreak
\item If $G$ is solvable and non-trivial, then $F(G)$ is non-trivial. 

\medbreak
\item If $N \lhd G$, then  $N\cap F(G)=F(N)$. 
In particular, if $\Rad(G)$ is non-trivial, then  $F(G)$ is non-trivial.
\end{enumerate}
\end{rem}

\begin{rem}\label{rem:ringforrep}
\begin{enumerate}[leftmargin=*,label=\rm{(\alph*)}]
\item\label{item:order1} Let $V$ be a finite-dimensional simple $\fie G$-module. Then there are a subring $R$ of $\fie$ with $R^\times \cap \ndZ=\{-1,1\} $ and a projective $RG$-submodule $V_R$ of $V$ such that $V\simeq \fie\ot_R V_R$. Indeed, let $v\in V \backslash \{0\}$ and let $B\subseteq \{gv\mid g\in G\}$ be a 
$\fie$-basis of $V$.
Then there is a number field $L\subseteq \fie$ such that $gv\in LB$ for all $g\in G$. The ring of integers $R=\Oc_L$ of $L$ is a Dedekind domain with $R^\times \cap \ndZ=\{-1,1\}$, and $RGv=\sum _{g\in G}Rgv$ is a finitely generated torsion-free (hence projective) $RG$-submodule of $V$ with $\fie\ot_R RGv\simeq V$.

\medbreak
\item Let $V$ be a simple Yetter-Drinfeld module over $\fie G$. Then there are a subring $R$ of $\fie$ with $R^\times \cap \ndZ=\{-1,1\} $ and an $R$-order $V_R$ of $V$. Indeed, there exist $g\in G$ and a simple $\fie C_G(g)$-module $U$ such that $V\simeq M(g,U)$. By 
\ref{item:order1}, there exist a subring $R$ of $\fie$ and a projective $RC_G(g)$-module $U_R$ such that $U\simeq \fie\ot_R U_R$.
 Then $V_R = M(g,U_R)$ is an $R$-order of $V$.
\end{enumerate}
\end{rem}

\begin{rem}
\label{rem:roots_of_1}
Let $p$ be a prime and $R$ be a subring of $\fie$ not containing $p^{-1}$. Let $q\in R$ be a primitive root of $1$ of order $mp^l$, where $\gcd(m,p)=1$ and $l\ge 0$. Let $\fm $ be a maximal ideal of $R$ containing $p$ and $\ndF = R/\fm$. 
Since the polynomial $t^m-1\in \ndF[t]$ is separable, it follows that $\mathrm{ord} (q\ot_R \ndF)=m$, and there exists a root of unity $\zeta \in q^{\ndZ}\subseteq R$ of order $m$ such that $q\zeta ^{-1}$ is a primitive root of 1 of order $p^l$.
\end{rem}

We can now prove the main result of this paper.

\begin{thm}
\label{thm:char_zero}
Let $G$ be a finite group such that $\Rad(G)\ne Z^*(G)$.
Let $V$ be a Yetter-Drinfeld module over $\fie G$ such that 
$\supp V$ is a conjugacy class of $G$ generating $G$, and that $\dim\NA(V)<\infty$.
Then $V$ is absolutely simple  
and isomorphic to one of the Yetter-Drinfeld modules
of Examples~\ref{exa:affine}, \ref{exa:T}, \ref{exa:S4a} and \ref{exa:S4b}. 
\end{thm}

\begin{proof}
Since $\Rad(G)\ne Z^*(G)$, Remark~\ref{rem:Fitting} implies 
that $|F(G/Z^*(G))|>1$. 
Let $p$ be a prime divisor of $\left|F\left(G/Z^*(G) \right) \right|$. 
Since $\supp V$ is a conjugacy class of $G$, $V\simeq  M(g,\rho)$ for a unique representation $\rho$ of $C_G(g)$. Since $\dim V<\infty$, 
$V$ is absolutely simple by Corollary~\ref{cor:absolutely_irreducible}. By Remark \ref{rem:ringforrep}, there is a subring $R$ of $\fie$ with $p\notin R^\times $ and a Yetter-Drinfeld order $V_R$ of $V$. Let $g\in \supp V$ and let 
\[
\rho_R\colon RC_G(g)\to \Aut ( (V_R)_g),
\quad h\mapsto (v\mapsto hv).
\]
Then $V_R\simeq  M(g,\rho_R)$.
Let $\fm$ be a maximal ideal of $R$ such that $p \in \fm$, 
and let $\ndF = R/ \fm$ and 
$V_{\ndF}= \ndF\ot _R V_R$. 
Then $V_{\ndF}$ is a Yetter-Drinfeld module over $\ndF G$ with support $\supp V$. Since $\NA(V)$ is  finite-dimensional, so is $\NA (V_{\ndF})$  by Lemma~\ref{lem:orders}.
Then Theorem~\ref{thm:char_p} applies:
$V_{\ndF}$ is absolutely simple 
and isomorphic to one of the Yetter-Drinfeld modules
of Examples~\ref{exa:affine}, \ref{exa:T}, \ref{exa:S4a} and \ref{exa:S4b}.
In particular, \[
\dim V_g=\dim (V_{\ndF})_g=1.
\]

If $|\supp V|\in\{3,5,7\}$, then the claim 
follows from \cite[Theorem 1.6]{MR4729697}.

Assume now that $|\supp V|=4$. We may assume $V$ has
basis $v_1,v_2,v_3,v_4$, where $\deg v_i=x_i$ for all
$i\in\{1,2,3,4\}$, and the action of $\fie G$ on $V$ is given by Table~\ref{tab:T}, where $\epsilon\in\{-1,1\}$ and 
$q\in \fie $ is such that 
$\ndF \ot  _R q^3=1_{\ndF}$.
By Remark~\ref{rem:roots_of_1}, there exist $\zeta ,r\in R$ such that $\zeta^3=1$, $q=\zeta r$, and $r^{2^m}=1$ for some $m\geq0$. 

There are five cases to consider:
\begin{enumerate}[leftmargin=*,label=\rm{(\alph*)}]

 \item If $m\geq2$ or $r=\epsilon$, then 
  $\dim\NA(V)=\infty$ by Lemma~\ref{lem:T}. 

\medbreak
 \item If $r=-\epsilon$, $\epsilon=1$, and $\zeta=1$, then  $V$ is isomorphic to the Yetter-Drinfeld module of Example~\ref{exa:T} with
 $(\epsilon,q)=(1,-1)$. 

 \medbreak
 \item If $r=-\epsilon$, $\epsilon=-1$, and $\zeta=1$, then $q=1$ and $\fie[v_1]$ is an infinite-dimensional subalgebra of $\NA (V)$. 
 
 \medbreak
 \item If $r=-\epsilon$, $\epsilon=1$, and $\zeta^3=1$, $\zeta\ne 1$, then $\dim\NA(V)=\infty$ by Lemma \ref{lem:Tcubic} for $p=3$. 
 
 \medbreak
 \item If $r=-\epsilon$, $\epsilon=-1$, and $\zeta^3=1$, and $\zeta\ne 1$, then $V$ is isomorphic to the Yetter-Drinfeld module of Example~\ref{exa:T} with
 $(\epsilon,q)=(-1,\zeta)$. 
\end{enumerate}

\medbreak
Finally, assume that $|\supp V|=6$. Here either $\supp V=\Oc_2^4$ or $\supp V=\Oc_4^4$. In both cases, the cocycle corresponding to $V_{\ndF}$ is constant $1$, and $\car \ndF =2$. Let $g\in \supp V$.
By Remark~\ref{rem:quandles6},
$C_G(g)$ is generated by $g$ and the unique element $h\in \supp V$ with $gh=hg$ and $h\ne g$.
By Remark~\ref{rem:roots_of_1},
there exist $q,\epsilon \in \fie $ such that $q$ is a root of one of order a power of two, and $\epsilon^4=1$.
Moreover, if $\supp V=\Oc_2^4$, then $\epsilon^2=1$.
There are two cases to consider depending on the rack. 

\smallbreak
Assume first 
that $\supp V=\Oc_4^4$. By Lemma \ref{lem:4cyclesS4}, since 
$\dim\NA(V)<\infty$, 
$q=-1$ and $\epsilon=1$. Thus $V$ is isomorphic 
to the Yetter-Drinfeld module of Example~\ref{exa:S4b}. 

\smallbreak
Assume now that $\supp V=\Oc_2^4$. Then 
$V_g\oplus V_h$ is a braided subspace of $V$ of diagonal
type with braiding matrix
\[
\begin{pmatrix}
q & \epsilon q\\
\epsilon q & q
\end{pmatrix}.
\]
This braiding is of Cartan type with symmetric Cartan matrix. 
Hence, by \cite[Theorem 1]{MR2207786}, 
$\dim\NA(V_g\oplus V_h)<\infty$ if and only if
$(q^2-1)(q^3-1)=0$ and $q\ne 1$. Since the order of $q\in 2^{\ndN}$ and $\dim \NA (V)<\infty$, we see  
that $q=-1$ and $\epsilon\in\{-1,1\}$. Thus 
$V$ is isomorphic to one of the Yetter-Drinfeld modules
of Example~\ref{exa:S4a}. 
\end{proof}

\begin{lem} 
\label{lem:nilpotent}
A nilpotent group $G$ generated by  a conjugacy class $X$ is 
 cyclic.
\end{lem}

\begin{proof}
  We proceed by induction on the nilpotency class of $G$. If $G$ is abelian, then $|X|=1$ and $G$ is cyclic.
  Otherwise let $Z$ be the center of $G$. Then $G/Z$ is generated by $\{xZ\mid x\in X\}$. By induction hypothesis, $G/Z$ is cyclic. It follows that $G$ is abelian and hence $|X|=1$ and $G$ is cyclic.
\end{proof}

\begin{thm}\label{thm:char_zero-solvable}
Let $G$ be a finite solvable non-cyclic group, and 
$V \in \yd{\fie G}$ such that 
$\supp V$ is a conjugacy class of $G$ generating $G$ and $\dim\NA(V)<\infty$. Then $V$ is absolutely simple 
and isomorphic to one of the Yetter-Drinfeld modules
of Examples~\emph{\ref{exa:affine}, \ref{exa:T}, \ref{exa:S4a}} and \emph{\ref{exa:S4b}}. 
In particular, $\vert G \vert$ is an even integer and $\gcd(\vert G \vert, 105) \neq 1$.
\end{thm}

\begin{proof}
 Theorem~\ref{thm:char_zero} applies because $\Rad G=G \ne Z^*(G)$ by Lemma~\ref{lem:nilpotent}. 
\end{proof}

\begin{cor}\label{cor:affine-indecomposable}
Let $G$ be a finite group and let $X$ be a conjugacy 
class of $G$ which as a rack is affine, indecomposable, and non-trivial.
Let $V \in \yd{\fie G}$ such that $\supp V = X$. 
If $X$ is not in \eqref{eq:list-affine-simple-racks}, then 
$\dim \NA(V) = \infty$.
\end{cor}

\begin{proof} We may assume that $X$ generates $G$.
The hypothesis says that $X\simeq  \Aff(\varGamma,\aut)$ where $\varGamma$
is a finite non-trivial abelian group  and $\aut$ is a fixed-point free automorphism  of $\varGamma$, see Remark~\ref{rem:affine_indec}.
Then $G$ is solvable by Corollary \ref{cor:abelian_centralizers} \ref{item:abelian-cetralizers},
but non-abelian because $\varGamma \neq \{e\}$. 
Assume that $\dim \NA(V) < \infty$. Then Theorem~\ref{thm:char_zero-solvable}
says that either $X$ is in \eqref{eq:list-affine-simple-racks} or 
$X \simeq \cO_2^4$ or $X \simeq \cO_4^4$. In the last two possibilities, 
$\Inn X \simeq \SG{4}$ but  since $[\SG{4},\SG{4}]=\AG{4}$ is not nilpotent, Proposition~\ref{pro:affine-derived}\ref{item:affine-derived-2} implies that $X$ is not affine.
Thus 
$X$ is in \eqref{eq:list-affine-simple-racks}.
\end{proof}

\subsection{Simple racks}\label{subsec:simple-racks-appl}
Simple racks were classified in \cite{MR1994219} and \cite{MR0682881}; those 
that can be realized as conjugacy classes of finite groups are of two kinds:

\begin{enumerate}[leftmargin=*,label=\rm{(\Roman*)}]
\item\label{item:simple-affine} simple racks whose order is a power of a prime,

\medbreak
\item\label{item:simple-non-abelian} simple racks whose order is divisible by at least two primes.
\end{enumerate}

\medbreak
Simple racks in class \ref{item:simple-non-abelian} 
are not affine and are related to finite non-abelian simple groups; for instance conjugacy classes in such groups are in class
\ref{item:simple-non-abelian}. Theorem~\ref{thm:char_zero} does not apply immediately to Yetter-Drinfeld modules over simple groups.

\medbreak
Simple racks with $q = p^n$ elements ($p$ prime), i.e., in class \ref{item:simple-affine},
are affine. Precisely, given such a rack $X$, there exists $T \in GL(n, \ndF_p)$ 
whose minimal polynomial $f \in \ndF_p[t]$ is monic irreducible, such that 
$X \simeq \Aff(\ndF_p^n,T) \simeq \Aff(\ndF_q, d)$
where $f$ is the minimal polynomial of $d \in \ndF_q$. 
From the description in \cite{MR1994219} we easily conclude:

\begin{rem}\label{rem:simple-solvable}
Let $X$ be a simple rack that can be realized as a conjugacy class of a finite group. Then $\Inn X$ is solvable if and only if $X$ is of class \ref{item:simple-affine}. 
\end{rem}

\begin{proof}[Proof of Theorem~\ref{thm:list-affine-simple-racks}]
Since simple racks are indecomposable, Theorem~\ref{thm:list-affine-simple-racks}
follows from Corollary~\ref{cor:affine-indecomposable}.
\end{proof}

\subsection{Nichols algebras over groups with odd order}

Recall from \cite[Definition 2.3]{MR3713037} that a finite rack
$X$ is \emph{of type C}  when there are  a decomposable subrack
$Y = R\coprod S$  and elements $r\in R$, $s\in S$ such that
\begin{align}
\label{eq:typeC-rack-inequality}
&r \triangleright s \neq s \qquad (\text{hence } s \triangleright r \neq r),
\\ \label{eq:typeC-rack-indecomposable}
&R = \Oc^{\Inn  Y}_{r}, \qquad S = \Oc^{\Inn  Y}_{s},
\\
\label{eq:typeC-rack-dimension} &\min \{\vert R \vert, \vert S \vert \}  > 2 \quad \text{ or } \quad
\max \{\vert R \vert, \vert S \vert \}  > 4.
\end{align}
The next result is a handy consequence of \cite[Theorem 2.1]{MR3656477},
cf. Proposition~\ref{pro:consequence}.

\begin{lem}\label{th:typeC} \cite[2.9]{MR3713037}
If $\Oc$ is a finite rack of type C, then
$\dim\NA(\Oc, \bq) = \infty$ for every finite faithful 2-cocycle $\bq$. \qed
\end{lem}

The proof of the next result combines \cite[Theorem 2.1]{MR4732368}
with Theorem \ref{thm:char_zero-solvable}.

\begin{thm}\label{thm:solvable-210-order-abelian-or-typeC} 
Let $G$ be a group of odd order and let $V \in \yd{\fie G}$. 
If $\NA(V)$ is finite-dimensional,
then $X = \supp V$ is an abelian rack. If additionally $\fie$ is algebraically closed, 
then $V$ is of diagonal type.
\end{thm}

\begin{proof} 
Assume that $X$ is  not abelian, i.e., there exist $r,s \in X$ such that $r\triangleright s \neq s$. Let $L = \langle r,s \rangle \leq G$,
$R= \Oc_r^{L}$, $S= \Oc_s^{L}$ and $Y= R\cup S$; 
then $\vert R \vert \neq 1$ and 
$\vert S \vert \neq 1$.

\medbreak
If $R\neq S$, 
then  $Y$ is a decomposable subrack of $X$
that satisfies \eqref{eq:typeC-rack-indecomposable}, because  $L = \langle Y\rangle$ so 
$R=\Oc_r^{L}=\Oc_r^{\Inn  Y}$ and $\Oc_s^{L}=\Oc_s^{\Inn  Y}$.  
Now \eqref{eq:typeC-rack-inequality} holds by assumption. 
Let $p, q$ be primes such that $p$ divides $\vert R \vert$ and $q$
divides $\vert S \vert$. Since $\vert R \vert$ and $\vert S \vert$ divide
$\vert L \vert$ and fortiori $\vert G \vert$, 
  $\min \{\vert R \vert, \vert S \vert \} \geq  \min \{p, q \} \geq 3$, thus 
\eqref{eq:typeC-rack-dimension} holds. Hence $X$ is of type C
and $\dim \NA(V) = \infty$.

\medbreak
Finally, if $s \in \Oc_r^{L}$, then $ Y = \Oc_r^{L}$ and $L = \langle Y \rangle$.
Let $W = \bigoplus_{g \in Y} V_G$; clearly $W \in \yd{\fie L}$.
Now $G$ and $L$ are solvable by the Feit-Thompson Theorem, and $L$ is not abelian by assumption.
Hence $\dim \NA(W) = \infty$ by Theorem \ref{thm:char_zero-solvable}.
Since $W$ is a braided subspace of $V$, we conclude that $\dim \NA(V) = \infty$.
\end{proof}

The Nichols algebras appearing in Theorem~\ref{thm:solvable-210-order-abelian-or-typeC} can be described more explicitly and are well studied.

\begin{rem}\label{rem:diagonal-odd-order}
Assume that $\fie$ is algebraically closed.
Let $G$ be a group of odd order and let $V \in \yd{\fie G}$ such that
$\dim \NA(V) <\infty$; thus
$V$ is of diagonal type by Theorem~\ref{thm:solvable-210-order-abelian-or-typeC}. 
Let $(q_{ij})_{1 \le i,j \le \dim V}$ be the braiding matrix
of $V$; then all $q_{ij}$'s  have odd order. 
By inspection of the Tables in \cite{MR2462836} there are three possibilities:
\begin{enumerate}[leftmargin=*,label=\rm{(\roman*)}]
\item  
$V$ is of Cartan type, hence  $\dim \NA(V)$ is known to be odd. 
The defining relations are described in \cite[Theorem 6.2]{MR2630042}.

\medbreak
\item 
The Dynkin diagram of $V$ is in row 5 of Table 1 in \cite{MR2462836}; it 
depends on two parameters, 
$\zeta$ of order 3 and $q\notin \{0,1,\zeta,\zeta^2\}$:
\begin{center}
    \Dchaintwo{$\zeta$}{$q^{-1}$}{$q$}
\end{center}
Set $N = \ord q$, 
$M = \ord \zeta q^{-1}$. Then $\dim \NA(V) = 3^2MN$ is odd, whenever $N$ is odd,
see \cite[\S\, 7.2]{MR3736568}.

\medbreak
\item The Dynkin diagram of $V$ is in \cite[row 18, Table 2]{MR2462836}, namely 
\begin{center}
\Dchainthree{}{$\zeta$}{$\zeta^{-1}$}{$\zeta$}{$\zeta^{-1}$}{$\zeta^{-3}$}
\quad and\quad  
\Dchainthree{}{$\zeta$}{$\zeta^{-1}$}{$\zeta^{-4}$}{$\zeta^4$}{$\zeta^{-3}$}
\end{center}
where $\zeta$ is a primitive 9th root of one. 
Here $\dim \NA(V) = 3^{22}$, see \cite[\S\, 7.3]{MR3736568}.
\end{enumerate}
In particular, $\dim \NA(V)$ is odd in all cases.

 Note that
all character values of $G$ are roots of unity of odd order, since $G$ has odd order.
Therefore, the odd-dimensionality of $\NA (V)$ already follows from the structure of Kharchenko's PBW basis, see \cite[Definition~7, Theorem~2]{MR1763385}.
\end{rem}

\subsection*{Acknowledgements}
The work of N.~A. was partially supported by CONICET (PIP 11220200102916CO), FONCyT-ANPCyT (PICT-2019-03660), by the Secyt (UNC), by the International Center of Mathematics, Southern University
of Science and Technology, Shenzhen and by the National Science Foundation under Grant No. DMS-1928930, while he was in residence at the SLMath Research Institute in Berkeley, California, during July 2024.

The work of L.~V. was partially supported by
the project OZR3762 of Vrije Universiteit Brussel and
FWO Senior Research Project G004124N.

The authors thank
the Mathematisches Forschungsinstitut Oberwolfach for providing a productive environment during the
Mini-Workshop ``Bridging Number Theory and Nichols Algebras via Deformations'' in February 2024, where this project was initiated.

\bibliographystyle{abbrv}
\bibliography{refs.bib}

\end{document}